\documentclass[11pt,letterpaper]{amsart}
\usepackage[margin=1.15in]{geometry}
\usepackage[T1]{fontenc}
\usepackage{lmodern}
\usepackage{microtype}

\usepackage{amsmath,amssymb,amsthm,mathtools}
\numberwithin{equation}{section}
\usepackage{enumitem}

\usepackage[hidelinks]{hyperref}

\hypersetup{
  pdftitle={Canonical Mandelbrot Cascades on Curves Are Rajchman},
  pdfauthor={Yin Cai, Guozheng Cheng, Xiang Fang, Menghan Li,
  Hongdou Qu, Chengbo Xiao},
  pdfsubject={Rajchman decay for canonical dyadic Mandelbrot cascades on intervals and planar curves},
  pdfkeywords={Mandelbrot cascades, Rajchman measures, Fourier decay, Fourier dimension, nonvanishing curvature, size-biased spine}
}

\newtheorem{theorem}{Theorem}[section]
\newtheorem{proposition}[theorem]{Proposition}
\newtheorem{lemma}[theorem]{Lemma}
\newtheorem{corollary}[theorem]{Corollary}
\newtheorem{fact}[theorem]{Fact}
\theoremstyle{definition}
\newtheorem{definition}[theorem]{Definition}

\theoremstyle{remark}
\newtheorem{remark}[theorem]{Remark}

\newcommand{\E}{\mathbb{E}}
\newcommand{\R}{\mathbb{R}}
\newcommand{\T}{\mathbb{T}}

\newcommand{\dimF}{\dim_{\mathrm F}}

\newcommand{\dimH}{\dim_{\mathrm H}}

\newcommand{\boldparagraph}[1]{%
  \paragraph{\textbf{#1}}%
}

\title{Canonical Mandelbrot Cascades on Curves Are Rajchman}

\author[Y. Cai]{Yin Cai}
\address{School of Mathematics\\
Hangzhou Normal University\\
Hangzhou 311121\\
P. R. China}
\email{cymath@hznu.edu.cn}

\author[G. Cheng]{Guozheng Cheng}
\address{School of Mathematical Sciences\\
Dalian University of Technology\\
Dalian 116024\\
P. R. China}
\email{gzhcheng@dlut.edu.cn}

\author[X. Fang]{Xiang Fang}
\address{Department of Applied Mathematics\\
National Yang Ming Chiao Tung University\\
Hsinchu 30010\\
Taiwan}
\email{xfang@nycu.edu.tw}

\author[M. Li]{Menghan Li}
\address{School of Mathematical Sciences\\
Dalian University of Technology\\
Dalian 116024\\
P. R. China}
\email{lmh2025@mail.dlut.edu.cn}

\author[H. Qu]{Hongdou Qu}
\address{School of Mathematical Sciences\\
Dalian University of Technology\\
Dalian 116024\\
P. R. China}
\email{quhongdou0926@mail.dlut.edu.cn}

\author[C. Xiao]{Chengbo Xiao}
\address{School of Mathematical Sciences\\
Dalian University of Technology\\
Dalian 116024\\
P. R. China}
\email{xiaochengbo@mail.dlut.edu.cn}



\date{}

\begin{document}
\begin{abstract}
We settle the Rajchman problem for canonical scalar dyadic Mandelbrot cascades at the minimal Kahane--Peyri\`ere integrability threshold.
If $\mu$ is the cascade on $[0,1]$, 
then $\widehat{\mu}(\xi)\to 0$ as $|\xi|\to\infty$, 
almost surely on non-extinction.  
For every fixed nondegenerate $C^2$ embedded arc $\gamma:[0,1]\to\mathbb{R}^2$, 
the pushforward $\gamma_\#\mu$ is likewise Rajchman almost surely on non-extinction.  
The analogous conclusion holds for the scalar cascade on the parameter circle pushed forward 
by any fixed nondegenerate \(C^2\) Jordan curve. 
No moment condition of order strictly greater than one is imposed; in particular, the results include the regime \(\mathbb{E}[W^q]=\infty\) for every \(q>1\).

The proof combines a spine-based lower-deviation principle, 
adaptive terminal approximation, 
and predictable capping to obtain almost-sure estimates uniform over large frequency annuli without higher moments.  
For curved pushforwards, an endpoint-safe phase decomposition controls direction-dependent stationary regions, 
including those meeting the endpoints of an arc, and couples the geometric and probabilistic arguments through a common dyadic kernel.  
Combined with the exact Fourier-dimension formulas for the corresponding models, 
the theorems show that Rajchman decay persists at zero Fourier dimension.
\end{abstract}

\maketitle

\tableofcontents

\section{Introduction and main results}\label{S:introduction}

\subsection{The historical problem and its qualitative endpoint}
\label{SS:introduction-history}

Mandelbrot cascades are canonical random multifractal measures, 
introduced as models of intermittent mass redistribution in turbulence.  
In the scalar dyadic model, independent copies of a nonnegative weight are attached to the binary tree, 
and mass is redistributed multiplicatively from one generation to the next.  
The associated positive branching martingale makes local mass exponents and
dimension theory accessible.  Fourier decay probes a different feature:
global oscillatory cancellation.  Such cancellation is not determined by
dimensional information alone and, at the minimal Kahane--Peyri\`ere
integrability threshold, must be obtained without higher-moment estimates.
This tension between tree multiplicativity and harmonic cancellation is the source of the Fourier problem for cascades; 
see \cite{Kahane1985SomeRandomSeries,Kahane1987PositiveMartingales, KahanePeyriere1976,Mandelbrot1974} 
for the foundations.

The Fourier problem originates in Mandelbrot's 1976 study of intermittent turbulence.  
He asked whether the Fourier coefficients of the limiting cascade obey a spectral power law
\begin{equation}\label{E:intro-Mandelbrot-power-law}
        |\widehat\mu(k)|^2\sim k^{-D}
        \qquad (k\to\infty)
\end{equation}
for an appropriate exponent \(D\), 
and how the optimal spectral exponent is related to the Hausdorff dimension of the measure.  
In describing the harmonic-analytic background, 
he wrote that ``the nature of the relationship between \(D\) and Fourier analysis has long been central''
\cite[p.~127]{Mandelbrot1976}.  
Kahane later placed Fourier decay among the basic questions for natural random measures \cite{Kahane1993}.

Viewed in this way, Mandelbrot's quantitative question contains a logically prior qualitative one: 
before asking for the optimal spectral exponent, 
one may ask whether the Fourier transform of a canonical cascade tends to zero at all.  
In modern terminology this is the Rajchman problem; 
for historical background on Rajchman measures and their role in Fourier analysis, see \cite{Lyons1995Rajchman}.  
It becomes a genuine endpoint problem in regimes where no positive uniform power law is available.

To formulate the distinction precisely, let \(\nu\) be a finite Borel measure on \(\R^d\).  
We use the convention
\begin{equation}\label{E:intro-Fourier-transform}
        \widehat\nu(\xi)
        =
        \int_{\R^d}e^{-2\pi i\xi\cdot x}\,d\nu(x),
\end{equation}
and define
\begin{equation}\label{E:intro-Fourier-dimension}
\dimF\nu
=
\sup\left\{
0\leq s\leq d:
|\widehat\nu(\xi)|=O(|\xi|^{-s/2})
\text{ as }|\xi|\to\infty
\right\}.
\end{equation}
For general background on Hausdorff dimension, Fourier dimension, 
and Fourier transforms of fractal measures, see \cite{Mattila2015}. 
Motivated by the modern formulation of the Mandelbrot--Kahane problem in
\cite{ChenHanQiuWang2024,ChenHanQiuWang2025CR}, 
it is useful here to distinguish four related but logically separate questions:
\begin{enumerate}[label=\textup{(\roman*)}]
\item \emph{Rajchman decay:} does
\(
\widehat\nu(\xi)\to0
\)
as \(|\xi|\to\infty\)?

\item \emph{Polynomial Fourier decay:} does there exist \(s>0\)
such that
\(
|\widehat\nu(\xi)|=O(|\xi|^{-s/2})
\)?

\item \emph{Exact Fourier dimension:} what is the value of \(\dimF\nu\)?

\item \emph{The Salem property:} when does
\(
\dimF\nu=\dimH\nu
\)?
\end{enumerate}
Positive polynomial decay clearly implies the Rajchman property. 
The converse from Rajchman decay to any positive power law is false in general.  
Thus Rajchman decay is not merely a weak form of positive Fourier dimension: 
if \(\dimF\nu=0\), then every positive uniform power law is ruled out, 
but the qualitative possibility \(\widehat\nu(\xi)\to0\) remains open.  
We call a measure \emph{pure Rajchman} if it is Rajchman and has Fourier dimension zero.

\boldparagraph{Recent dimension theory and the remaining endpoint.}
Recent work has substantially advanced the quantitative side of this program.
For canonical interval cascades, Fourier-dimension results were first obtained
under additional tail or moment assumptions \cite{ChenHanQiuWang2024,ChenLiSuomala2025}.  
Exact formulas for the canonical interval and circle models were subsequently established 
under the minimal Kahane--Peyri\`ere assumptions \cite{CaiChengFangLiQuXiao2026}.  
For cascades on planar curves, the first Fourier-dimension results required stronger tail hypotheses \cite{RyouSuomala2026}, 
while the corresponding minimal-integrability formulas were obtained later in \cite{CaiFangQu2026Curves}.

The exact formulas, together with their immediate equality cases, 
resolve questions \textup{(ii)}--\textup{(iv)} within the four-question framework adopted here for the canonical scalar models.  
The sole remaining issue is Rajchman decay at the zero-Fourier-dimension endpoint.  
In these models, the dimension formulas identify that endpoint precisely with the heavy-tail regime, 
characterized by \(\E[W^q]=\infty\) for every \(q>1\).  
The theorems below establish Rajchman decay under the minimal Kahane--Peyri\`ere assumptions 
and thereby
complete the picture for the remaining qualitative endpoint for these canonical models.

\boldparagraph{Other Rajchman models.}
The Rajchman property has also been established in several other classes of measures.  
Spatially independent martingales, including fractal-percolation models, 
may exhibit rapid or polynomial Fourier decay \cite{ShmerkinSuomala2018}.  
Renewal and arithmetic methods give broad Rajchman criteria, 
as well as descriptions of exceptional non-Rajchman configurations, 
for self-similar measures \cite{Bremont2021,LiSahlsten2022,Rapaport2022}; 
irreducibility and non-compactness yield Fourier decay for important self-affine systems \cite{LiSahlsten2020}. 
Nonlinear pushforwards and self-conformal measures form another closely related direction.  
Polynomial Fourier decay for broad classes of fractal measures under nonlinear maps was established in \cite{BakerBanaji2025},
while planar analytic self-conformal measures were treated under suitable nonlinearity and irreducibility assumptions in \cite{AlgomRodriguezHertzWang2026}. 
On the random side, Gaussian multiplicative chaos provides another major class of random multifractal measures; 
see \cite{RhodesVargas2014} for a general review. 
Subcritical Gaussian multiplicative chaos on the circle is almost surely Rajchman \cite{GarbanVargas2026}. 
These results provide essential context, but their proofs use spatial independence, 
renewal or algebraic structure, matrix irreducibility, or Gaussian input.  
None directly controls a general scalar Mandelbrot cascade under  the minimal Kahane--Peyri\`ere assumptions.

\subsection{Curvature, uniformity, and heavy tails}
\label{SS:introduction-difficulties}

There are three distinct obstructions at the Rajchman endpoint: 
a geometric obstruction coming from normal frequencies, 
an annular uniformity problem in the conclusion itself, and a probabilistic obstruction caused by heavy tails.
The interval theorem is the canonical one-dimensional endpoint result.  
The curved pushforward theorems add a genuinely geometric difficulty rather than a formal change of variables.

\boldparagraph{The affine-line obstruction.}
Let \(\gamma(t)=a+tv\) be an affine embedding, let \(\mu\) be a nonzero finite positive measure on \([0,1]\), and set \(\nu=\gamma_\#\mu\).  
For every unit vector \(\omega\perp v\) and every \(R>0\),
\begin{equation}\label{E:intro-affine-obstruction}
\widehat\nu(R\omega)
=
\int e^{-2\pi iR\omega\cdot(a+tv)}\,d\mu(t)
=
e^{-2\pi iR\omega\cdot a}\,\mu([0,1]).
\end{equation}
Thus a nonzero positive measure carried by a line is never Rajchman as a measure on \(\R^2\), 
even when its parameter measure is Rajchman on \(\R\).
Nonvanishing curvature removes this fixed normal-frequency obstruction, 
but it does not by itself prove Fourier decay.  
It replaces a permanent resonance by a moving stationary region whose location depends on the direction of the frequency.

For a fixed curved parametrization,
\begin{equation}\label{E:intro-curve-phase}
\widehat{\gamma_\#\mu}(r\omega)
=
\int_0^1 e^{-2\pi ir\,\omega\cdot\gamma(t)}\,d\mu(t),
\qquad \omega\in\mathbb S^1.
\end{equation}
The frequency has both a radial variable and a direction, 
and the stationary regions of \(t\mapsto\omega\cdot\gamma(t)\) move with \(\omega\).  
For an arc, such regions may enter or leave the parameter interval through either endpoint. 
Moreover, the amplitude is a cascade measure rather than a smooth density, 
so the usual smooth-amplitude stationary-phase theorem is not directly available.  
The curve theorem therefore requires a direction-uniform phase decomposition that is compatible 
with random cascade mass and with endpoint stationary regimes.

\boldparagraph{Uniformity over all frequencies.}
Rajchman decay is not a fixed-frequency statement.  
It requires one probability-one event on which all sufficiently large frequencies are controlled simultaneously.  
Writing
\begin{equation}\label{E:intro-annular-supremum}
M_n(\nu)
:=
\sup_{2^n\leq|\xi|<2^{n+1}}|\widehat\nu(\xi)|,
\end{equation}
the desired conclusion is \(M_n(\nu)\to0\) almost surely.  
A fixed-frequency estimate, or even decay along a prescribed lacunary sequence, 
does not imply this annular statement.  
Compact support gives a Lipschitz modulus in \(\xi\), so an annulus can be discretized; 
in the planar case a unit-scale grid has cardinality \(O(2^{2n})\).  
The exceptional probability at one grid point must therefore absorb this entropy and remain summable in \(n\), 
so that Borel--Cantelli yields a single samplewise tail estimate.

Weak convergence of the finite cascade approximations does not remove this uniformity problem.  
The convergence \(\mu_N\Rightarrow\mu\) implies \(\widehat{\mu_N}(\xi)\to\widehat\mu(\xi)\) for each fixed \(\xi\), 
but the Lipschitz norm of the test function \(t\mapsto e^{-2\pi i\xi\cdot\gamma(t)}\) grows with \(|\xi|\).  
Thus weak convergence alone gives no approximation uniform in frequencies whose size grows with the cascade level.  
Large terminal cylinders and their descendant masses must instead be controlled 
on the same event and uniformly over the frequency grid.

\boldparagraph{The heavy-tail obstruction.}
The third difficulty is probabilistic.  
The minimal Kahane--Peyri\`ere assumptions allow the genuinely heavy-tailed regime
\begin{equation}\label{E:intro-extreme-heavy-tail}
        \E[W^q]=\infty
        \qquad\text{for every }q>1.
\end{equation}
This lies outside the finite-\((1+\delta)\)-moment arguments that yield polynomial Fourier decay and outside the available finite-moment annular machinery.  
At the same time, the endpoint obstruction is not simply the absence of all concentration.  
After the relevant martingale increments are predictably capped at size \(b/n\), 
Freedman's inequality \cite{Freedman1975} gives bounds of the form
\begin{equation}\label{E:intro-Freedman-scale}
        C\exp(-c n/b),
\end{equation}
which, for sufficiently small fixed \(b>0\), absorb the planar grid entropy.
The centered capped martingale is therefore not the main endpoint obstruction.

What remains is the uncapped positive large-jump contribution and its predictable tail compensator.  
Without any \(L^{1+\eta}\) information,
standard moment estimates do not make these terms summable uniformly in the frequency annulus and in the tree generation.  
Thus the endpoint problem is simultaneously geometric, annular, and large-jump probabilistic; 
the proof must control all three features on a single almost-sure event.  
The condition \eqref{E:intro-extreme-heavy-tail} is not an additional hypothesis of the main theorems, 
but it marks the regime in which the pure Rajchman phenomenon is most transparent.

\subsection{Standing assumptions and main theorems}
\label{SS:introduction-main-results}

Throughout the paper, the scalar cascade weight \(W\) satisfies
\begin{equation}\label{E:intro-KP}
W\geq0,
\qquad
\E W=1,
\qquad
\E[W\log_2^+W]<\infty,
\qquad
\E[W\log_2W]<1,
\end{equation}
where \(0\log_2 0=0\).  We write
\begin{equation}\label{E:intro-chi}
        \chi=1-\E[W\log_2W]>0.
\end{equation}
These are the minimal Kahane--Peyri\`ere assumptions for nontriviality in the subcritical case.  
In particular, no condition of the form
\(
\E[W^{1+\eta}]<\infty
\)
is imposed.  
Unless explicitly stated otherwise, all almost-sure assertions below are understood on the corresponding non-extinction event.

Our first theorem is the one-dimensional endpoint result.

\begin{theorem}[Interval Rajchman theorem]
\label{T:main-interval-rajchman}
Let \(\mu\) be the canonical scalar dyadic Mandelbrot cascade on \([0,1]\) generated by \(W\).  
Then, almost surely on \(\{\mu([0,1])>0\}\),
\begin{equation}\label{E:intro-interval-conclusion}
        \widehat\mu(\xi)\longrightarrow0
        \qquad (|\xi|\to\infty,\ \xi\in\R).
\end{equation}
\end{theorem}

\begin{remark}
\label{R:intro-further-extensions}
Theorem~\ref{T:main-interval-rajchman} is stated in the scalar dyadic setting in order to keep the formulation and the endpoint mechanism in their simplest form.  
Vector-weight and \(b\)-adic variants in higher dimensions require additional formulation and bookkeeping and are not pursued here.
\end{remark}

The main geometric result asserts that this qualitative decay persists under every prescribed nondegenerate curved embedding of the parameter interval into the plane.

\begin{definition}[Fixed nondegenerate embedded arc] 
\label{D:main-nondegenerate-arc}
A map \(\gamma:[0,1]\to\R^2\) is a fixed nondegenerate \(C^2\) embedded arc if it is a \(C^2\) embedding and
\begin{equation}\label{E:intro-arc-nondegeneracy}
\inf_{t\in[0,1]}|\gamma'(t)|>0,
\qquad
\inf_{t\in[0,1]}
|\det(\gamma'(t),\gamma''(t))|>0.
\end{equation}
\end{definition}

\begin{theorem}[Fixed-arc Rajchman theorem]
\label{T:main-arc-rajchman}
Let \(\mu\) be as in Theorem~\ref{T:main-interval-rajchman}, 
and let \(\gamma:[0,1]\to\R^2\) be a fixed nondegenerate \(C^2\) embedded arc.  
Then, almost surely on \(\{\mu([0,1])>0\}\),
\begin{equation}\label{E:intro-arc-conclusion}
        \widehat{\gamma_\#\mu}(\xi)\longrightarrow0
        \qquad (|\xi|\to\infty,\ \xi\in\R^2).
\end{equation}
\end{theorem}

There is a corresponding theorem for closed curves.  
Here the parametrization is part of the data, since the cascade is defined on the parameter circle \(\T=\R/\mathbb Z\).

\begin{definition}[Fixed nondegenerate  Jordan curve]
\label{definition:fixed-curve}
A map \(\Gamma:\mathbb T\to\mathbb R^2\) is called a fixed nondegenerate
\(C^2\) Jordan curve if \(\Gamma\) is a \(C^2\) embedding and satisfies
 \begin{equation}\label{E:intro-Jordan-nondegeneracy}
        \inf_{t\in\mathbb T}|\Gamma'(t)|>0,
     \qquad
    \inf_{t\in\mathbb T}
     |\det(\Gamma'(t),\Gamma''(t))|>0.
\end{equation}
 When derivatives are written, \(\Gamma\) is represented by its associated
 \(1\)-periodic \(C^2\) lift to \(\mathbb R\).
 \end{definition}

No invariance under nonlinear reparametrizations is asserted;
the theorem applies separately to every fixed parametrization satisfying
\eqref{E:intro-Jordan-nondegeneracy}.

\begin{theorem}[Fixed-Jordan-curve Rajchman theorem]
\label{T:main-jordan-rajchman}
Let \(\widetilde\mu^{\T}\) be the canonical scalar dyadic cascade on \(\T\)
generated by \(W\), and let \(\Gamma:\T\to\R^2\) be a fixed nondegenerate
\(C^2\) Jordan curve.  Then, almost surely on
\(\{\widetilde\mu^{\T}(\T)>0\}\),
\begin{equation}\label{E:intro-Jordan-conclusion}
\widehat{\Gamma_\#\widetilde\mu^{\T}}(\xi)\longrightarrow0
\qquad (|\xi|\to\infty,\ \xi\in\R^2).
\end{equation}
\end{theorem}

The Salem equality cases follow directly from the exact formulas via the equality condition
in Jensen's inequality under the size-biased law.  Consequently, the exact Fourier-dimension
results of \cite{CaiChengFangLiQuXiao2026,CaiFangQu2026Curves}, together with
Theorems~\ref{T:main-interval-rajchman},~\ref{T:main-arc-rajchman}, and~\ref{T:main-jordan-rajchman},
complete the Fourier-decay picture described above for the canonical scalar interval
and fixed-curve models under the minimal Kahane--Peyri\`ere assumptions.

The extreme heavy-tail regime makes the distinction especially transparent.

\begin{corollary}[Pure Rajchman cascades]
\label{C:main-pure-rajchman}
Assume in addition that
\begin{equation}\label{E:intro-no-higher-moments}
        \E[W^q]=\infty
        \qquad\text{for every }q>1.
\end{equation}
Then, almost surely on the corresponding non-extinction events, 
the interval cascade \(\mu\), each prescribed fixed nondegenerate arc pushforward \(\gamma_\#\mu\), 
and each prescribed fixed nondegenerate Jordan-curve pushforward \(\Gamma_\#\widetilde\mu^{\T}\)
are Rajchman and have Fourier dimension zero.
\end{corollary}

The additional hypothesis in Corollary~\ref{C:main-pure-rajchman} is not used to prove Rajchman decay.  
It is used only in the exact Fourier-dimension formulas to identify a regime in which all positive power-decay exponents vanish. 
Thus the conclusion of the main theorems is genuinely qualitative:
at this endpoint the Fourier transform tends to zero although the Fourier dimension is zero.

\subsection{Ideas of the proof}
\label{SS:introduction-proof-ideas}

The proof consists of two components joined by a common dyadic kernel.  
The first is a probabilistic endpoint mechanism for scalar cascades under only the Kahane--Peyri\`ere assumptions; 
it is independent of the phase and yields the interval theorem as well as the random estimates needed for curves.  
The second is a deterministic, endpoint-safe decomposition of the nonlinear phase on a nondegenerate arc.  
The coefficients produced by this decomposition have exactly the kernel controlled by the probabilistic mechanism.  
This matching is what allows the argument to remain effective at the zero-exponent endpoint,
where no polynomial estimate is available.

\boldparagraph{A spine-gauge lower-deviation principle.}
For a word \(u\in\{0,1\}^n\), let
\begin{equation}\label{E:intro-Au-Vu}
A_u=2^{-n}\prod_{m=1}^{n}W_{u|m},
\qquad
V(u)=-\log_2A_u,
\end{equation}
with \(V(u)=+\infty\) if \(A_u=0\).  For \(0<\theta<\chi\), define the
lower-deviation mass
\begin{equation}\label{E:intro-Ln}
        L_n(\theta)
        =
        \sum_{|u|=n}
        A_u\,\mathbf{1}_{\{V(u)\leq\theta n\}}.
\end{equation}
Under the size-biased law, \(\chi\) is the typical linear growth rate of \(V\).  
Thus \(L_n(\theta)\) is the generation-\(n\) mass carried by lower-deviation cylinders, 
equivalently those with \(A_u\geq2^{-\theta n}\).  
The first key estimate is the qualitative limit
\begin{equation}\label{E:intro-Ln-limit}
        L_n(\theta)\longrightarrow0
        \qquad\text{almost surely for every }0<\theta<\chi.
\end{equation}
Its proof combines a size-biased spine, complete convergence for the spine random walk on dyadic blocks, 
a many-to-one identity, and a hereditary amplification argument.  
The use of size-biased changes of measure on trees is classical; 
see, for example, \cite{LyonsPemantlePeres1995}. 
In place of the polynomial lower-deviation bounds available under higher moments, 
\eqref{E:intro-Ln-limit} shows directly that abnormally large cylinders carry asymptotically negligible total mass.

The bare convergence in \eqref{E:intro-Ln-limit} is not by itself stable under the sums over generations arising in a Fourier expansion.  
The form needed there is the following shifted convolution consequence: 
for fixed \(h\geq0\) and \(C_0<\infty\),
\begin{equation}\label{E:intro-shifted-convolution}
\sup_{|m-n|\leq C_0}
\sum_{k=1}^{n+h}
\min\{1,2^{k-m}\}L_k(\theta)
\longrightarrow0
\qquad\text{almost surely}.
\end{equation}
The bounded shift \(|m-n|\leq C_0\) is essential: 
it permits a purely qualitative lower-deviation estimate to survive the accumulation over all generations.  
Formula \eqref{E:intro-shifted-convolution} is the point at which the spine gauge becomes compatible with oscillatory analysis.

\boldparagraph{Terminal approximation and a capped martingale skeleton.}
A separate issue is the passage from a finite cascade level to the limiting measure.  
Because weak convergence is not uniform at growing frequencies, 
we introduce an adaptive terminal cutoff.  
Large terminal cylinders are discarded at a threshold of order \(1/n\), 
and both their realized descendant mass and their predictable compensator are shown to vanish using exact dimensionality,
the spine estimate, and only \(L^1\)-integrability of the terminal mass.

The remaining finite-level oscillatory expansion is arranged as a stopped martingale skeleton.  
Its increments are predictably capped at size \(b/n\).
Freedman's inequality then gives the exponential bound \eqref{E:intro-Freedman-scale}, 
which is strong enough, for a suitably small fixed \(b>0\), to absorb the entropy of the planar frequency grid.  
The uncapped remainder is separated into two positive terms: 
the realized large jumps and their predictable tail compensator.  
The former are controlled by \eqref{E:intro-shifted-convolution}; 
the latter are controlled by the same lower-deviation mechanism together with
\(
\E[W\mathbf{1}_{\{W>t\}}]\to0
\).
Thus the proof isolates three analytically distinct contributions---centered capped fluctuations, 
realized large jumps, and predictable compensators---and applies concentration only after predictable capping has removed the need for a higher moment.

\boldparagraph{Endpoint-safe phase geometry.}
The deterministic component treats the nonlinear phase in
\eqref{E:intro-curve-phase}.  For a nondegenerate arc, the tangent angle
\(\Theta_\gamma\) satisfies
\begin{equation}\label{E:intro-tangent-angle}
        \Theta_\gamma'(t)
        =
        \frac{\det(\gamma'(t),\gamma''(t))}{|\gamma'(t)|^2},
\end{equation}
and is therefore monotone and bi-Lipschitz.  
Consequently, the derivative sublevel sets of \(t\mapsto\xi\cdot\gamma(t)\) have geometry that is uniform in the direction of \(\xi\), 
even when a stationary region meets an endpoint of the parameter interval.

At annular scale \(2^n\leq|\xi|<2^{n+1}\), we construct a partition consisting of two endpoint pieces, 
a small-derivative piece, and dyadic derivative bands.
For a fixed derivative cutoff, the endpoint and small-derivative pieces,
together with the bands below the oscillatory cutoff, 
form a safe region contained in a bounded union of intervals of length
\begin{equation}\label{E:intro-safe-length}
        O_\gamma(2^{-j}+2^{-n/2}).
\end{equation}
Its contribution is controlled by atomlessness and a weak-transfer lemma,
without any polynomial local-mass estimate.  
This controls moving stationary regions, including endpoint stationary regimes, at the qualitative endpoint.

Only \(O_{\gamma,j}(1)\) oscillatory bands remain.  
Their phase-bin scales \(m_{\xi,d}\) satisfy \(|m_{\xi,d}-n|\leq C_{\gamma,j}\).  
If \(\chi_{\xi,d}\) denotes the corresponding localization, 
integration by parts on a dyadic interval \(J\) of generation \(\ell+1\) gives
\begin{equation}\label{E:intro-coefficient-envelope}
\left|
2^\ell\int_J
 e^{-2\pi i\xi\cdot\gamma(t)}\chi_{\xi,d}(t)\,dt
\right|
\leq
C_\gamma\min\{1,2^{\ell-m_{\xi,d}}\}.
\end{equation}
This is exactly the kernel in \eqref{E:intro-shifted-convolution}.  
The phase geometry and the spine estimate therefore meet without loss of scale: 
the number of active bands is bounded, 
and every active phase scale remains within bounded distance of the annular scale.

\boldparagraph{Annular assembly.}
The probabilistic mechanism is first assembled on one-dimensional annular grids, 
together with the adaptive terminal cutoff and the frequency Lipschitz bound, 
to prove the interval theorem.  
For a curved pushforward, the safe-region estimate and the stopped-skeleton estimate on the oscillatory bands are combined on a deterministic planar grid.  
The capped concentration bound pays for the grid entropy, 
while the lower-deviation estimates remove the uncapped and terminal remainders.  
Borel--Cantelli then gives a single almost-sure estimate on all sufficiently large grids, 
and the frequency Lipschitz bound recovers the full annuli.  
This proves the fixed-arc theorem.
The Jordan-curve theorem follows by cutting the parameter circle at the first dyadic generation and applying the arc argument to the two descendant interval cascades.

The proof is therefore not obtained by taking a zero-exponent limit of an existing polynomial estimate.  
It builds an annular Fourier-decay mechanism whose essential inputs are qualitative and which remains valid at the zero-Fourier-dimension endpoint.

\subsection{Scope and organization}
\label{SS:introduction-scope}

In the curve theorems, the parametrized curve is fixed before the cascade is sampled.  
For each prescribed arc or Jordan curve satisfying the stated nondegeneracy hypotheses, 
the Rajchman conclusion holds almost surely on non-extinction.  
We do not claim a single probability-one event valid simultaneously for all curves in a class; 
the exceptional null set and the constants in the proof may depend on the chosen parametrization.  
The parametrization in the Jordan-curve theorem is part of the model, not merely a choice of notation: the dyadic cascade is constructed on
the parameter circle, and a nonlinear reparametrization generally changes the pushforward measure. Accordingly, the theorem is stated separately for each fixed parametrization satisfying the stated nondegeneracy conditions.

The conclusions are qualitative, and this is sharp with respect to power-law decay.  
We prove that the relevant Fourier transforms vanish at infinity without asserting a positive polynomial exponent.  
In the regime of Corollary~\ref{C:main-pure-rajchman}, 
the resulting measures have Fourier dimension zero, so no positive uniform power law can hold.  
The qualitative endpoint is therefore intrinsic to the problem rather than an artifact of the proof.

The remainder of the paper is organized according to the separation between the probabilistic and geometric components described above.  
Section~\ref{S:preliminaries} sets up the dyadic-tree and cascade notation, descendant decompositions, 
Fourier localization, and curve geometry.  
Section~\ref{S:spine-engine} develops the spine-gauge and stopped-martingale mechanisms and proves the interval theorem.  
Section~\ref{S:arc-geometry} establishes the endpoint-safe phase decomposition for fixed nondegenerate arcs, 
and Section~\ref{S:arc-assembly} combines it with the probabilistic estimates to prove the fixed-arc theorem.  
Finally, Section~\ref{S:jordan-cutting} proves the Jordan-curve theorem by first-generation dyadic cutting of the parameter circle 
and concludes by combining the three Rajchman theorems with the exact Fourier-dimension formulas to establish Corollary~\ref{C:main-pure-rajchman}.

\section{Cascade preliminaries and notation}\label{S:preliminaries}

This section collects the cascade notation and the auxiliary facts used throughout the paper.  
The interval cascade on \([0,1]\) is the primary object.  
A separate copy on
\[
        \mathbb T=\mathbb R/\mathbb Z
\]
is introduced for the Jordan-curve reduction.

Unless explicitly stated otherwise, the cascade weight \(W\) satisfies
\[
        W\geq0,
        \qquad
        \mathbb EW=1,
        \qquad
        \mathbb E[W\log_2^+W]<\infty,
        \qquad
        \mathbb E[W\log_2W]<1,
\]
where
\[
        \log_2^+x=\max\{\log_2x,0\}
        =\log_2(\max\{x,1\}),
        \qquad
        0\log_20=0.
\]
We write
\[
        \chi=1-\mathbb E[W\log_2W]>0.
\]
No moment assumption of the form
\(
        \mathbb E[W^{1+\eta}]<\infty
\)
is imposed.

\subsection{The dyadic tree and the interval cascade}
\label{SS:preliminaries-tree}

Let \(\mathbb N_0=\{0,1,2,\ldots\}\) and
\[
        \{0,1\}^*=\bigcup_{n\in\mathbb N_0}\{0,1\}^n
\]
be the rooted binary tree, whose root is the empty word \(\varnothing\).
For
\[
        u=(u_1,\ldots,u_n)\in\{0,1\}^n,
\]
write \(|u|=n\), and for \(0\leq k\leq n\) set
\[
        u|k=(u_1,\ldots,u_k),
        \qquad
        u|0=\varnothing.
\]
Concatenation of words is denoted by juxtaposition.

For \(u=(u_1,\ldots,u_n)\in\{0,1\}^n\), define
\[
        a_u=\sum_{j=1}^n u_j2^{-j},
        \qquad
        a_\varnothing=0.
\]
Set \(I_\varnothing=[0,1]\), and for \(|u|=n\geq1\) let
\[
        I_u=
        \begin{cases}
        [a_u,a_u+2^{-n}),&u\neq(1,\ldots,1),\\
        [a_u,1],&u=(1,\ldots,1).
        \end{cases}
\]
For each \(n\in\mathbb N_0\), define
\[
\mathcal{D}_n([0,1])
=
\{I_u:|u|=n\}.
\]
Then \(\mathcal{D}_n([0,1])\) is a partition of \([0,1]\) into dyadic intervals of length \(2^{-n}\).
If \(J=I_u\in\mathcal D_{n+1}([0,1])\), its dyadic parent is
\[
J^-=I_{u|n}.
\]
For \(t\in[0,1]\), let \(I_n(t)\) be the unique member of \(\mathcal D_n([0,1])\) containing \(t\) under this endpoint convention.
The convention is immaterial on the full-probability event on which the limiting cascade assigns zero mass to all dyadic endpoints.

Attach independent copies
\[
        \{W_u:u\in\{0,1\}^*,\ |u|\geq1\}
\]
of \(W\) to the non-root vertices.  
For \(|u|=n\), define
\[
        A_u=2^{-n}\prod_{k=1}^n W_{u|k},
        \qquad
        A_\varnothing=1.
\]
The level-\(n\) cascade measure is
\[
        \mathrm d\mu_n(t)
        =
        \sum_{|u|=n}2^nA_u\mathbf 1_{I_u}(t)\,\mathrm dt,
\]
so that
\[
        \mu_n(I_u)=A_u
        \qquad(|u|=n).
\]
Its total mass is
\[
        Z_n=\mu_n([0,1])=\sum_{|u|=n}A_u.
\]
Then \((Z_n)_{n\geq0}\) is a nonnegative martingale with \(\mathbb EZ_n=1\).

The weak convergence of the cascade measures, the \(L^1\)-convergence of the total-mass martingale, 
and its non-degeneracy are classical consequences of the Kahane--Peyri\`ere theory; 
see \cite{Kahane1985SomeRandomSeries,Kahane1987PositiveMartingales,KahanePeyriere1976}.

\begin{fact}\label{F:preliminaries-standard-cascade-facts}
Under the assumptions above, there is a finite Borel measure \(\mu\) on \([0,1]\) 
and a nonnegative random variable \(Z_\infty\) such that:
\begin{enumerate}
\item \(\mu_n\Rightarrow\mu\) weakly almost surely;

\item \(Z_n\to Z_\infty\) almost surely and in \(L^1\), and
\[
        \mu([0,1])=Z_\infty,
        \qquad
        \mathbb{E}Z_\infty=1;
\]

\item the non-extinction event
\[
        \mathcal S
        =
        \{Z_\infty>0\}
        =
        \{\mu([0,1])>0\}
\]
has positive probability;

\item on \(\mathcal S\), the measure \(\mu\) is exact-dimensional of dimension \(\chi\); 
more precisely,
\[
        \lim_{n\to\infty}
        -\frac{1}{n}\log_2\mu(I_n(t))
        =
        \chi
\]
for \(\mu\)-almost every \(t\); 
see \cite[Theorem~12(1)]{Barral2014Mandelbrot};

\item consequently, since \(\chi>0\), the measure \(\mu\) is atomless on \(\mathcal S\).
\end{enumerate}
\end{fact}

\subsection{Descendant cascades and lower-deviation notation}
\label{SS:preliminaries-descendants}

For a finite word \(u\), the descendant environment below \(u\) is
\[
        \{W_{uv}:v\in\{0,1\}^*,\ |v|\geq1\}.
\]
For \(v\in\{0,1\}^m\), define
\[
        A_v^{(u)}
        =2^{-m}\prod_{\ell=1}^m W_{u(v|\ell)},
        \qquad
        A_\varnothing^{(u)}=1,
\]
and
\[
        \mathrm d\mu_m^{(u)}(t)
        =
        \sum_{|v|=m}2^mA_v^{(u)}\mathbf 1_{I_v}(t)\,\mathrm dt.
\]
Its total mass is
\[
        Z_m^{(u)}
        =\mu_m^{(u)}([0,1])
        =\sum_{|v|=m}A_v^{(u)}.
\]
Simultaneously for every finite word \(u\), almost surely,
\[
        \mu_m^{(u)}\Rightarrow\mu^{(u)},
        \qquad
        Z_m^{(u)}\longrightarrow
        Z_\infty^{(u)}:=\mu^{(u)}([0,1]).
\]
For each fixed generation \(n\), the family
\(
\{Z_\infty^{(u)}:|u|=n\}
\)
consists of independent copies of \(Z_\infty\) and is independent of the weights up to generation \(n\).  
On the simultaneous convergence and endpoint-null event,
\[
        \mu(I_u)=A_uZ_\infty^{(u)}
        \qquad(u\in\{0,1\}^*).
\]

For \(u\in\{0,1\}^*\), set
\[
        V(u)=
        \begin{cases}
        -\log_2A_u,&A_u>0,\\
        +\infty,&A_u=0,
        \end{cases}
\]
so that \(A_u=2^{-V(u)}\), with \(2^{-\infty}=0\).  
For \(0<\theta<\chi\), define the lower-deviation mass
\[
        L_n(\theta)
        =
        \sum_{|u|=n}
        A_u\mathbf 1_{\{V(u)\leq\theta n\}}.
\]
Section~\ref{S:spine-engine} proves that \(L_n(\theta)\to0\) almost surely for every \(0<\theta<\chi\).

\subsection{The circle cascade and first-generation cutting}
\label{SS:preliminaries-circle-cascade}

Let \(m_{\mathbb T}\) be normalized Haar probability measure on \(\mathbb T\).  
For \(|u|=n\), define the dyadic arc
\[
        I_u^{\mathbb T}
        =\{t+\mathbb Z:t\in[a_u,a_u+2^{-n})\}.
\]
On a possibly enlarged probability space, let
\[
        \{W_u^{\mathbb T}:u\in\{0,1\}^*,\ |u|\geq1\}
\]
be independent copies of \(W\), independent also of the interval-cascade environment.  
Set
\[
        A_u^{\mathbb T}
        =2^{-n}\prod_{k=1}^n W_{u|k}^{\mathbb T},
        \qquad
        A_\varnothing^{\mathbb T}=1,
\]
and
\[
        \mathrm d\widetilde\mu_n^{\mathbb T}(x)
        =
        \sum_{|u|=n}
        2^nA_u^{\mathbb T}
        \mathbf 1_{I_u^{\mathbb T}}(x)\,
        \mathrm dm_{\mathbb T}(x).
\]
Then
\[
        \widetilde Z_n^{\mathbb T}
        :=\widetilde\mu_n^{\mathbb T}(\mathbb T)
        =\sum_{|u|=n}A_u^{\mathbb T}
\]
is a nonnegative mean-one martingale.  
By the same cascade convergence theorem,
\[
        \widetilde\mu_n^{\mathbb T}
        \Rightarrow
        \widetilde\mu^{\mathbb T},
        \qquad
        \widetilde Z_n^{\mathbb T}
        \longrightarrow
        \widetilde Z_\infty^{\mathbb T}
        =\widetilde\mu^{\mathbb T}(\mathbb T)
\]
almost surely, with convergence of the total masses also in \(L^1\).  
The circle non-extinction event is
\[
        \mathcal S^{\mathbb T}
        =\{\widetilde Z_\infty^{\mathbb T}>0\}.
\]

For \(i\in\{0,1\}\), let
\[
        \rho_0(t)=\frac t2,
        \qquad
        \rho_1(t)=\frac{1+t}{2}\pmod1,
        \qquad t\in[0,1].
\]
The descendant environment below the first-generation vertex \(i\) defines an interval cascade \(\mu^{\mathbb T,(i)}\).  
More explicitly, for \(v\in\{0,1\}^m\), set
\[
        A_v^{\mathbb T,(i)}
        =2^{-m}\prod_{\ell=1}^m
        W_{i(v|\ell)}^{\mathbb T},
\]
and let \(\mu^{\mathbb T,(i)}\) be the almost-sure weak limit of
\[
        \mathrm d\mu_m^{\mathbb T,(i)}(t)
        =
        \sum_{|v|=m}
        2^mA_v^{\mathbb T,(i)}
        \mathbf 1_{I_v}(t)\,\mathrm dt.
\]
The measures \(\mu^{\mathbb T,(0)}\) and \(\mu^{\mathbb T,(1)}\) are independent copies in law of \(\mu\), 
and are independent of \(W_0^{\mathbb T}\) and \(W_1^{\mathbb T}\).

\begin{lemma}[First-generation cutting identity]
\label{L:first-generation-cutting}
Almost surely, for each \(i\in\{0,1\}\),
\[
        \widetilde\mu^{\mathbb T}\big|_{I_i^{\mathbb T}}
        =
        \frac{W_i^{\mathbb T}}2
        (\rho_i)_\#\mu^{\mathbb T,(i)}.
\]
Consequently, for every continuous map
\(F:\mathbb T\to\mathbb R^2\),
\[
        F_\#\widetilde\mu^{\mathbb T}
        =
        \frac{W_0^{\mathbb T}}2
        (F\circ\rho_0)_\#\mu^{\mathbb T,(0)}
        +
        \frac{W_1^{\mathbb T}}2
        (F\circ\rho_1)_\#\mu^{\mathbb T,(1)}.
\]
\end{lemma}

\begin{proof}
For \(i\in\{0,1\}\), \(m\geq0\), and
\(v\in\{0,1\}^m\), one has
\[
        A_{iv}^{\mathbb T}
        =
        \frac{W_i^{\mathbb T}}2
        A_v^{\mathbb T,(i)}.
\]
Moreover, \(I_{iv}^{\mathbb T}\) and \(\rho_i(I_v)\) differ only at dyadic endpoints. 
Since the level measures are absolutely continuous,  it follows that, for every \(\varphi\in C(\mathbb T)\),
\[
        \int_{I_i^{\mathbb T}}
        \varphi\,\mathrm d\widetilde\mu_{m+1}^{\mathbb T}
        =
        \frac{W_i^{\mathbb T}}2
        \int_{[0,1]}
        \varphi\circ\rho_i\,
        \mathrm d\mu_m^{\mathbb T,(i)}.
\]

On the probability-one event on which the relevant weak convergences hold and
\[
        \widetilde\mu^{\mathbb T}
        \bigl(\partial I_i^{\mathbb T}\bigr)
        =
        0,
        \qquad i=0,1,
\]
we may let \(m\to\infty\). The right-hand side converges by weak convergence, 
since \(\varphi\circ\rho_i\) is continuous, 
while the left-hand side converges because \(I_i^{\mathbb T}\) is a continuity set for \(\widetilde\mu^{\mathbb T}\). 
Hence
\[
        \widetilde\mu^{\mathbb T}\big|_{I_i^{\mathbb T}}
        =
        \frac{W_i^{\mathbb T}}2
        (\rho_i)_\#\mu^{\mathbb T,(i)}.
\]
Summing over \(i=0,1\) and applying \(F_\#\) proves the second identity.
\end{proof}

\subsection{Curves and Fourier notation}
\label{SS:preliminaries-curves}

We use the notions of a fixed nondegenerate \(C^2\) embedded arc 
and a fixed nondegenerate \(C^2\) Jordan curve introduced in 
Definitions~\ref{D:main-nondegenerate-arc} and \ref{definition:fixed-curve}.  
For such an arc \(\gamma\) and such a Jordan curve \(\Gamma\), write
\[
        \mu_\gamma=\gamma_\#\mu,
        \qquad
        \mu_\Gamma^{\mathbb T}
        =\Gamma_\#\widetilde\mu^{\mathbb T}.
\]
All constants in fixed-curve statements may depend on the chosen parametrized curve; 
no uniformity over families of curves is asserted.

For \(i\in\{0,1\}\), let \(\Gamma_i=\Gamma\circ\rho_i\).  
Then
\[
        \Gamma_i'(t)=\frac12\Gamma'(\rho_i(t)),
        \qquad
        \Gamma_i''(t)=\frac14\Gamma''(\rho_i(t)),
\]
and hence
\[
        \det(\Gamma_i'(t),\Gamma_i''(t))
        =
        \frac18
        \det(\Gamma'(\rho_i(t)),\Gamma''(\rho_i(t))).
\]
Thus each \(\Gamma_i\) is a fixed nondegenerate embedded arc.

For a fixed arc \(\gamma\) and \(\xi\in\mathbb R^2\setminus\{0\}\), set
\[
        \phi_\xi(t)=-2\pi\xi\cdot\gamma(t).
\]
With the Fourier convention fixed in the introduction,
\[
        \widehat{\mu_\gamma}(\xi)
        =
        \int_0^1
        \mathrm e^{i\phi_\xi(t)}\,\mathrm d\mu(t).
\]
For \(n\geq1\), define
\[
        \mathcal A_n
        =\{\xi\in\mathbb R^2:2^n\leq|\xi|<2^{n+1}\},
        \qquad
        S_{\gamma,n}
        =\sup_{\xi\in\mathcal A_n}
        |\widehat{\mu_\gamma}(\xi)|.
\]
The fixed-arc Rajchman theorem is equivalent to
\[
        S_{\gamma,n}\longrightarrow0
\]
almost surely on \(\mathcal S\).

\subsection{Localization and small-interval transfer}
\label{SS:preliminaries-localization}

For \(K,K_1<\infty\), define
\[
        \Omega_K=\{Z_\infty\leq K\},
        \qquad
        \mathcal Z_{K_1}
        =\left\{\sup_{\ell\geq0}Z_\ell\leq K_1\right\}.
\]
Since \(Z_n\to Z_\infty<\infty\) almost surely,
\[
        \mathbb P\left(
        \bigcup_{K=1}^\infty
        \bigcup_{K_1=1}^\infty
        \Omega_K\cap\mathcal Z_{K_1}
        \right)=1.
\]

For a finite Borel measure \(\nu\) on \([0,1]\), define
\[
        \Delta_\rho(\nu)
        =
        \sup\{\nu(I):I\subset[0,1]\text{ is an interval and }|I|\leq\rho\}.
\]

\begin{lemma}\label{L:preliminaries-small-interval-modulus}
If \(\nu\) is atomless, then
\[
        \Delta_\rho(\nu)\longrightarrow0
        \qquad(\rho\downarrow0).
\]
If \(\nu_m\Rightarrow\nu\) weakly, then for every fixed \(\rho>0\),
\[
        \limsup_{m\to\infty}\Delta_\rho(\nu_m)
        \leq
        \Delta_{2\rho}(\nu).
\]
\end{lemma}

\begin{proof}
Suppose first that \(\nu\) is atomless.  
If the first conclusion failed, there would be \(\varepsilon>0\) and intervals \(I_m\) with \(|I_m|\to0\) and \(\nu(I_m)\geq\varepsilon\).  
After passing to a subsequence, the centers of \(I_m\) converge to some \(x\in[0,1]\).
Every neighborhood of \(x\) would then have \(\nu\)-mass at least \(\varepsilon\), contradicting atomlessness.

Now assume \(\nu_m\Rightarrow\nu\). 
Choose intervals \(I_m\) such that
\[
        |I_m|\leq\rho,
        \qquad
        \nu_m(I_m)\geq\Delta_\rho(\nu_m)-m^{-1}.
\]
Along a subsequence realizing the limsup, 
the two endpoints converge after passing to a further subsequence.  
For every \(\eta>0\), all sufficiently large \(I_m\) are contained in a fixed closed interval \(J_\eta\) of length at most \(\rho+\eta\).  
The portmanteau theorem gives
\[
        \limsup_{m\to\infty}\nu_m(I_m)
        \leq
        \limsup_{m\to\infty}\nu_m(J_\eta)
        \leq
        \nu(J_\eta)
        \leq
        \Delta_{\rho+\eta}(\nu).
\]
Taking \(\eta=\rho\) proves the second assertion.
\end{proof}

In the fixed-arc argument this lemma is applied to
\(\nu_m=\mu_m\) and \(\nu=\mu\).  
Fact~\ref{F:preliminaries-standard-cascade-facts} then makes the right-hand side small on non-extinction.

\subsection{Almost-sure and notation conventions}
\label{SS:preliminaries-full-event}

Let \(\mathcal E_0\) be the probability-one event on which the interval and circle cascades converge, 
all descendant limits indexed by finite words exist simultaneously, 
the cylinder and cutting identities above hold, 
and all dyadic endpoints have zero mass for the corresponding limiting measures on non-extinction.

The later proofs use only countably many integer or rational parameters and countably many deterministic frequency grids.  
We henceforth take the relevant countable intersections with \(\mathcal E_0\) without further comment.  
All resulting almost-sure assertions are therefore understood on a common probability-one event.

The endpoint-safe cutoffs, derivative bands, phase-bin scales, 
and oscillatory coefficients are introduced at their first use in Section~\ref{S:arc-geometry}; 
they are not part of the basic cascade notation.

We reserve
\(\mu_n,\ \mu,\ Z_n,\ Z_\infty\)
for the interval cascade, and
\(
        \widetilde\mu_n^{\mathbb T},\
        \widetilde\mu^{\mathbb T},\
        \widetilde Z_n^{\mathbb T},\
        \widetilde Z_\infty^{\mathbb T}
\)
for the circle cascade.  
Their geometric pushforwards are denoted by
\(\mu_\gamma\) and \(\mu_\Gamma^{\mathbb T}\), respectively.

\section{The spine-gauge Rajchman engine}\label{S:spine-engine}

This section develops the probabilistic framework used throughout the paper in five steps. 
First, a size-biased spine argument gives an almost-sure lower-deviation mass theorem.  
Second, this lower-deviation theorem implies a shifted convolution estimate adapted to fixed-cutoff phase scales.  
Third, an adaptive terminal cutoff and a terminal descendant theorem allow us to pass from the limiting cascade measure to level approximations.  
Fourth, an abstract stopped skeleton theorem controls the martingale arrays that arise from oscillatory phase decompositions.  
Fifth, these ingredients give the interval Rajchman theorem.

Throughout this section, let
\[
        \mathcal{F}_n
        :=
        \sigma\bigl(W_u:1\leq |u|\leq n\bigr),
        \qquad n\geq0,
\]
where \(\mathcal{F}_0=\{\varnothing,\Omega\}\).  For
\(J=I_u\in\mathcal{D}_k([0,1])\), write
\(
A_J:=A_u,
\)
\(
V(J):=V(u).
\)
If \(k\geq1\), also write \(W_J:=W_u\) and
\(
 J^{-}
        :=
        I_{u|k-1}
        \in
        \mathcal{D}_{k-1}([0,1])
\)
for the dyadic parent of \(J\).  Then
\(
        A_J
        =
        \frac{1}{2}A_{J^{-}}W_J.
\)

By Fact~\ref{F:preliminaries-standard-cascade-facts} and the descendant-cascade facts from
Subsection~\ref{SS:preliminaries-descendants}, 
there is a probability-one event on which \(\mu_n\Rightarrow\mu\) and \(Z_n\to Z_\infty<\infty\), 
and, simultaneously for every finite word \(u\in\{0,1\}^{\ast}\),
\[
        Z_m^{(u)}
        \to
        Z_\infty^{(u)}
        \qquad
        (m\to\infty),
\]
and, up to the harmless dyadic endpoint convention, \(\mu(I_u)=A_uZ_\infty^{(u)}\).
Moreover, \(\mathbb{E}Z_\infty=1\). 
For each fixed integer \(r\geq0\), the family \(\bigl\{Z_\infty^{(u)}:|u|=r\bigr\}\) consists of independent copies of \(Z_\infty\) and is independent of \(\mathcal{F}_r\).
Finally, almost surely on \(\{Z_\infty>0\}\), the measure \(\mu\) is exact-dimensional of dimension \(\chi>0\) and is therefore atomless.

\subsection{The size-biased spine and lower-deviation mass}\label{SS:spine-lower-deviation}

Let \(W^\star\) have the \(W\)-size-biased law characterized by
\[
        \mathbb{E}^{\star}[f(W^\star)]
        =
        \mathbb{E}[Wf(W)]
\]
for every nonnegative Borel function \(f\). 
Since \(W\geq0\) and \(\mathbb{E}W=1\), this is a probability law, and
\(\mathbb{P}^{\star}(W^\star=0)=0\).
Hence \(X:=1-\log_2W^\star\) is well defined \(\mathbb{P}^{\star}\)-almost surely, with
\[
        \mathbb{E}^{\star}|X|
        =
        \mathbb{E}\!\left[
        W|1-\log_2W|\mathbf{1}_{\{W>0\}}
        \right],
\]
because
\(
  W|1-\log_2W|\mathbf{1}_{\{W>0\}}
        \leq
        W+W\log_2^+W+C.
\)
Hence
\(
        \mathbb{E}^{\star}X
        =
        1-\mathbb{E}[W\log_2W]
        =
        \chi.
\)

Let \(X_1,X_2,\ldots\) be independent copies of \(X\) under \(\mathbb{P}^{\star}\), 
and write \(S_k^\star:=\sum_{i=1}^kX_i\). 
For \(0<\theta<\chi\) and \(N\geq1\), define
\[
        \Psi_\theta(N)
        :=
        \mathbb{P}^{\star}
        \left(
        \exists k\in[N,2N):
        S_k^\star\leq\theta k
        \right).
\]

\begin{lemma}\label{L:spine-dyadic-complete-convergence}
For every \(0<\theta<\chi\),
\[
        \sum_{j=1}^{\infty}\Psi_\theta(2^j)<\infty.
\]
\end{lemma}

\begin{proof}
Fix \(0<\theta<\chi\), and set \(\delta:=\chi-\theta>0\) and \(D_i:=\chi-X_i\).  
Then \(D_1,D_2,\ldots\) are independent and identically distributed, 
belong to \(L^1(\mathbb{P}^{\star})\), and have mean zero.  
If \(N\) is a positive dyadic integer and \(k\in[N,2N)\), 
then \(S_k^\star\leq\theta k\) implies
\[
        \sum_{i=1}^{k}D_i
        =
        \chi k-S_k^\star
        \geq
        (\chi-\theta)k
        \geq
        \delta N.
\]
Consequently,
\[
        \Psi_\theta(N)
        \leq
        \mathbb{P}^{\star}
        \left(
        \max_{1\leq k\leq2N}
        \sum_{i=1}^{k}D_i
        \geq \delta N
        \right).
\]
For each positive dyadic integer \(N\), define
\[
        Y_i^{(N)}
        :=
        D_i\mathbf{1}_{\{|D_i|\leq N\}},
        \qquad
        m_N
        :=
        \mathbb{E}^{\star}Y_1^{(N)}.
\]
Since \(\mathbb{E}^{\star}D_1=0\), we have
\(
        m_N
        =
        -\mathbb{E}^{\star}\!\left[
        D_1\mathbf{1}_{\{|D_1|>N\}}
        \right].
\)
Moreover, since \(D_1\in L^1(\mathbb{P}^{\star})\), dominated convergence yields
\[
        |m_N|
        \leq
        \mathbb{E}^{\star}\!\left[
        |D_1|\mathbf{1}_{\{|D_1|>N\}}
        \right]
        \longrightarrow0.
\]
Hence, for all sufficiently large dyadic \(N\),
\(
2N|m_N|
\leq
\frac{\delta N}{4}.
\)
Write
\(
\widetilde{Y}_i^{(N)}:=Y_i^{(N)}-m_N
\).
For fixed \(N\), the variables \(\widetilde{Y}_i^{(N)}\) are independent, centered, and square-integrable.  
Suppose that
\(
        \max_{1\leq k\leq2N}
        \sum_{i=1}^{k}D_i
        \geq \delta N
\)
and that \(|D_i|\leq N\) for every \(1\leq i\leq2N\).  
Then, for some \(k\leq2N\),
\[
        \sum_{i=1}^{k}\widetilde{Y}_i^{(N)}
        =
        \sum_{i=1}^{k}D_i-km_N
        \geq
        \delta N-2N|m_N|
        \geq
        \frac{3\delta N}{4}.
\]
It follows that, for all sufficiently large dyadic \(N\),
\[
        \Psi_\theta(N)
        \leq
        2N\mathbb{P}^{\star}(|D_1|>N)
        +
        \mathbb{P}^{\star}
        \left(
        \max_{1\leq k\leq2N}
        \sum_{i=1}^{k}\widetilde{Y}_i^{(N)}
        \geq
        \frac{3\delta N}{4}
        \right).
\]
By Kolmogorov's maximal inequality,
\[
\begin{aligned}
        \mathbb{P}^{\star}
        \left(
        \max_{1\leq k\leq2N}
        \sum_{i=1}^{k}\widetilde{Y}_i^{(N)}
        \geq
        \frac{3\delta N}{4}
        \right)
        &\leq
        \mathbb{P}^{\star}
        \left(
        \max_{1\leq k\leq2N}
        \left|
        \sum_{i=1}^{k}\widetilde{Y}_i^{(N)}
        \right|
        \geq
        \frac{3\delta N}{4}
        \right)
\\
        &\leq
        \frac{
        2N\operatorname{Var}^{\star}(Y_1^{(N)})
        }{
        (3\delta N/4)^2
        }
\\
        &\leq
        \frac{C_\delta}{N}
        \mathbb{E}^{\star}
        \left[
        D_1^2\mathbf{1}_{\{|D_1|\leq N\}}
        \right].
\end{aligned}
\]
Therefore, for all sufficiently large \(j\),
\[
        \Psi_\theta(2^j)
        \leq
        2^{j+1}\mathbb{P}^{\star}(|D_1|>2^j)
        +
        C_\delta2^{-j}
        \mathbb{E}^{\star}
        \left[
        D_1^2\mathbf{1}_{\{|D_1|\leq2^j\}}
        \right].
\]
By Tonelli's theorem,
\[
\begin{aligned}
        \sum_{j=1}^{\infty}
        2^j\mathbb{P}^{\star}(|D_1|>2^j)
        =
        \mathbb{E}^{\star}
        \left[
        \sum_{\substack{j\geq1\\2^j<|D_1|}}
        2^j
        \right]\leq
        2\mathbb{E}^{\star}|D_1|
        <\infty,
\end{aligned}
\]
where we used \(\sum_{\substack{j\geq1\\2^j<x}}2^j\leq2x\) for \(x\geq0\).  
Similarly,
\[
\begin{aligned}
        \sum_{j=1}^{\infty}
        2^{-j}
        \mathbb{E}^{\star}
        \left[
        D_1^2\mathbf{1}_{\{|D_1|\leq2^j\}}
        \right]=
        \mathbb{E}^{\star}
        \left[
        D_1^2
        \sum_{\substack{j\geq1\\2^j\geq|D_1|}}
        2^{-j}
        \right]\leq
        2\mathbb{E}^{\star}|D_1|
        <\infty,
\end{aligned}
\]
where we used
\(
        x^2
        \sum_{\substack{j\geq1\\2^j\geq x}}
        2^{-j}
        \leq
        2x
\)
for \(x\geq0\). 
Combining these estimates, and observing that the finitely many initial terms are bounded by \(1\), 
yields
\(
\sum_{j=1}^{\infty}\Psi_\theta(2^j)<\infty.
\)
\end{proof}

\begin{lemma}\label{L:spine-many-to-one}
For every integer \(n\geq1\) and every nonnegative Borel function
\(F:\mathbb{R}^n\to[0,\infty]\), one has
\[
        \mathbb{E}
        \left[
        \sum_{\substack{|u|=n\\A_u>0}}
        A_u
        F\bigl(V(u|1),\ldots,V(u|n)\bigr)
        \right]
        =
        \mathbb{E}^{\star}
        \left[
        F(S_1^\star,\ldots,S_n^\star)
        \right].
\]
Equivalently, one may sum over all \(|u|=n\), 
with the convention that the entire summand is \(0\) on \(\{A_u=0\}\).
\end{lemma}

\begin{proof}
Suppose first that \(F\) is bounded. On \(\{A_u>0\}\),
\(
V(u|j)=\sum_{m=1}^{j}(1-\log_2W_{u|m})
\)
for \(1\leq j\leq n\).  
Hence, by the definition of \(A_u\), independence of the weights along a fixed branch, and iteration of the size-biased identity,
\[
\begin{aligned}
        &\mathbb{E}
        \left[
        \sum_{\substack{|u|=n\\A_u>0}}
        A_u
        F\bigl(V(u|1),\ldots,V(u|n)\bigr)
        \right]
\\
        &\qquad=
        2^{-n}
        \sum_{|u|=n}
        \mathbb{E}
        \left[
        \prod_{m=1}^{n}W_{u|m}
        F
        \left(
        1-\log_2W_{u|1},
        \ldots,
        \sum_{m=1}^{n}
        \bigl(1-\log_2W_{u|m}\bigr)
        \right)
        \,;\,
        A_u>0
        \right]
\\
        &\qquad=
        2^{-n}
        \sum_{|u|=n}
        \mathbb{E}^{\star}
        \left[
        F(S_1^\star,\ldots,S_n^\star)
        \right]
\\
        &\qquad=
        \mathbb{E}^{\star}
        \left[
        F(S_1^\star,\ldots,S_n^\star)
        \right],
\end{aligned}
\]
where the last equality uses the fact that there are \(2^n\) words of length \(n\).

For a general nonnegative Borel function \(F\), apply the bounded case to \(F_M:=F\wedge M\) and let \(M\to\infty\). 
Monotone convergence on both sides completes the proof.
\end{proof}

\begin{lemma}\label{L:spine-hereditary-amplification}
Fix \(0<\theta<\chi\), and let \((R_j)_{j\geq1}\) be a deterministic sequence of positive numbers such that
\[
        \sum_{j=1}^{\infty}
        \frac{\Psi_\theta(2^j)}{R_j}
        <\infty.
\]
Then, almost surely, for all sufficiently large \(j\) and every integer \(n\) with \(2^j\leq n<2^{j+1}\), 
one has
\[
        L_n(\theta)\leq R_j.
\]
\end{lemma}

\begin{proof}
For each \(j\geq1\), put \(N:=2^j\).  For \(0\leq m\leq2N\) and \(u\in\{0,1\}^m\), define
\[
        B_{N,m}(u)
        :=
        \mathbf{1}_{\{
        \exists k\in[N,2N),\,
        k\leq m,\,
        V(u|k)\leq\theta k
        \}},
\]
and set
\[
        G_{N,m}(\theta)
        :=
        \sum_{|u|=m}A_uB_{N,m}(u).
\]
The variable \(G_{N,m}(\theta)\) is \(\mathcal{F}_m\)-measurable and satisfies \(0\leq G_{N,m}(\theta)\leq Z_m\), 
so it is integrable.

The indicators are hereditary along descendants: 
for \(0\leq m<2N\), \(u\in\{0,1\}^m\), and \(a\in\{0,1\}\),
\(
B_{N,m+1}(ua)\geq B_{N,m}(u)
\).
Using \(A_{ua}=\frac12A_uW_{ua}\), the independence of the generation-\((m+1)\) weights from \(\mathcal{F}_m\), 
and \(\mathbb{E}W=1\), we obtain
\[
\begin{aligned}
        \mathbb{E}
        \left[
        G_{N,m+1}(\theta)
        \mid\mathcal{F}_m
        \right]
        &=
        \sum_{|u|=m}
        \sum_{a\in\{0,1\}}
        \mathbb{E}
        \left[
        A_{ua}B_{N,m+1}(ua)
        \mid\mathcal{F}_m
        \right]
\\
        &\geq
        \sum_{|u|=m}
        B_{N,m}(u)
        \sum_{a\in\{0,1\}}
        \mathbb{E}
        \left[
        A_{ua}
        \mid\mathcal{F}_m
        \right]
\\
        &=
        \sum_{|u|=m}
        A_uB_{N,m}(u)
\\
        &=
        G_{N,m}(\theta).
\end{aligned}
\]
Thus
\(
(G_{N,m}(\theta),\mathcal{F}_m)_{0\leq m\leq2N}
\)
is a nonnegative submartingale.

Define the bounded Borel function
\[
        F_N(x_1,\ldots,x_{2N})
        :=
        \mathbf{1}_{\{
        \exists k\in[N,2N):
        x_k\leq\theta k
        \}}.
\]
The branches with \(A_u=0\) make no contribution.  
Applying Lemma~\ref{L:spine-many-to-one} gives
\[
\begin{aligned}
        \mathbb{E}G_{N,2N}(\theta)
        &=
        \mathbb{E}
        \left[
        \sum_{|u|=2N}
        A_u
        \mathbf{1}_{\{
        \exists k\in[N,2N):
        V(u|k)\leq\theta k
        \}}
        \right]
\\
        &=
        \mathbb{E}^{\star}
        \left[
        F_N(S_1^\star,\ldots,S_{2N}^\star)
        \right]
\\
        &=
        \mathbb{P}^{\star}
        \left(
        \exists k\in[N,2N):
        S_k^\star\leq\theta k
        \right)
\\
        &=
        \Psi_\theta(N).
\end{aligned}
\]
Doob's maximal inequality therefore yields
\[
\begin{aligned}
        \mathbb{P}
        \left(
        \max_{0\leq m\leq2N}
        G_{N,m}(\theta)>R_j
        \right)
        \leq
        \frac{\mathbb{E}G_{N,2N}(\theta)}{R_j}=
        \frac{\Psi_\theta(2^j)}{R_j}.
\end{aligned}
\]
By the assumed summability and the Borel--Cantelli lemma, almost surely, 
for all sufficiently large \(j\),
\[
        \max_{0\leq m\leq2^{j+1}}
        G_{2^j,m}(\theta)
        \leq
        R_j.
\]

Now let \(2^j\leq n<2^{j+1}\).  
If \(|u|=n\) and \(V(u)\leq\theta n\), then \(B_{2^j,n}(u)=1\), since \(k=n\) is admissible. 
Hence, for all sufficiently large \(j\),
\[
\begin{aligned}
        L_n(\theta)
        =
        \sum_{|u|=n}
        A_u\mathbf{1}_{\{V(u)\leq\theta n\}}\leq
        G_{2^j,n}(\theta)\leq
        R_j.
\end{aligned}
\]
This proves the lemma.
\end{proof}

\begin{theorem}[Spine lower-deviation mass]
\label{T:spine-lower-deviation}
For every \(0<\theta<\chi\),
\[
        L_n(\theta)
        =
        \sum_{|u|=n}
        A_u\mathbf{1}_{\{V(u)\leq\theta n\}}
        \longrightarrow0
        \qquad\text{almost surely}.
\]
Moreover, the convergence holds simultaneously for all \(0<\theta<\chi\) on a single probability-one event.
\end{theorem}

\begin{proof}
Fix \(0<\theta<\chi\).  
By Lemma~\ref{L:spine-dyadic-complete-convergence},
\(
\sum_{j=1}^{\infty}\Psi_\theta(2^j)<\infty
\).
Define
\[
        q_j
        :=
        \sum_{\ell=j}^{\infty}\Psi_\theta(2^\ell),
        \qquad
        R_j
        :=
        \sqrt{q_j}+2^{-j}.
\]
Since \(q_j\downarrow0\), we have \(R_j>0\) and \(R_j\to0\).
Moreover,
\(
        \sum_{j=1}^{\infty}
        \frac{\Psi_\theta(2^j)}{R_j}
        <\infty.
\)
Indeed, \(q_j-q_{j+1}=\Psi_\theta(2^j)\).  If \(q_j>0\), then
\[
\begin{aligned}
        \frac{\Psi_\theta(2^j)}{R_j}
        \leq
        \frac{q_j-q_{j+1}}{\sqrt{q_j}}
\leq
        2\bigl(\sqrt{q_j}-\sqrt{q_{j+1}}\bigr).
\end{aligned}
\]
If \(q_j=0\), then \(\Psi_\theta(2^j)=0\), so the corresponding summand vanishes.  
Therefore
\[
\begin{aligned}
        \sum_{j=1}^{\infty}
        \frac{\Psi_\theta(2^j)}{R_j}
        \leq
        2
        \sum_{j=1}^{\infty}
        \bigl(\sqrt{q_j}-\sqrt{q_{j+1}}\bigr)=
        2\sqrt{q_1}
        <\infty.
\end{aligned}
\]

Lemma~\ref{L:spine-hereditary-amplification} implies that, almost surely, 
for all sufficiently large \(j\) and every integer \(n\) with \(2^j\leq n<2^{j+1}\),
\(
L_n(\theta)\leq R_j
\).
Since \(R_j\to0\), it follows that \(L_n(\theta)\to0\).

Finally, intersect the corresponding probability-one events over all rational \(q\in(0,\chi)\).  
On the resulting probability-one event, 
let \(0<\theta<\chi\) be arbitrary and choose a rational \(q\in(\theta,\chi)\).  
Since
\(
L_n(\theta)\leq L_n(q)
\)
for every \(n\), and \(L_n(q)\to0\), we obtain \(L_n(\theta)\to0\).  
Thus the convergence holds simultaneously for all \(0<\theta<\chi\).
\end{proof}

\begin{lemma}[Shifted convolution]
\label{L:spine-shifted-convolution}
Let \((a_k)_{k\geq1}\) be a nonnegative sequence such that \(a_k\to0\) as \(k\to\infty\).  
Then, for every fixed \(h\in\mathbb{N}_0\) and \(C_0\geq0\),
\[
        \sup_{\substack{m\in\mathbb{Z}\\|m-n|\leq C_0}}
        \sum_{k=1}^{n+h}
        \min\{1,2^{k-m}\}a_k
        \longrightarrow0
        \qquad(n\to\infty).
\]
Consequently, on the probability-one event from Theorem~\ref{T:spine-lower-deviation}, 
for every \(0<\theta<\chi\),
\[
        \sup_{\substack{m\in\mathbb{Z}\\|m-n|\leq C_0}}
        \sum_{k=1}^{n+h}
        \min\{1,2^{k-m}\}L_k(\theta)
        \longrightarrow0
        \qquad(n\to\infty).
\]
\end{lemma}

\begin{proof}
Fix \(\eta>0\).  Since \(a_k\to0\), there exists \(K_0\geq1\) such that \(a_k\leq\eta\) for every \(k\geq K_0\).  
For all sufficiently large \(n\), we have \(n+h\geq K_0\) and \(n-C_0\geq1\).  
For such \(n\),
\[
\begin{aligned}
        \sum_{k=1}^{n+h}
        \min\{1,2^{k-m}\}a_k
        =
        \sum_{k=1}^{K_0-1}
        \min\{1,2^{k-m}\}a_k
+
        \sum_{k=K_0}^{n+h}
        \min\{1,2^{k-m}\}a_k.
\end{aligned}
\]

For the finite initial sum,
\[
\begin{aligned}
        \sup_{\substack{m\in\mathbb{Z}\\|m-n|\leq C_0}}
        \sum_{k=1}^{K_0-1}
        \min\{1,2^{k-m}\}a_k\leq
        2^{C_0-n}
        \sum_{k=1}^{K_0-1}2^ka_k
        \longrightarrow0.
\end{aligned}
\]

For the tail, if \(m\in\mathbb{Z}\) and \(|m-n|\leq C_0\), then
\[
\begin{aligned}
        \sum_{k=1}^{n+h}\min\{1,2^{k-m}\}
        \leq
        \sum_{k=1}^{m}2^{k-m}
        +
        \max\{n+h-m,0\}\leq
        2+h+C_0.
\end{aligned}
\]
Therefore,
\[
\begin{aligned}
        \sup_{\substack{m\in\mathbb{Z}\\|m-n|\leq C_0}}
        \sum_{k=K_0}^{n+h}
        \min\{1,2^{k-m}\}a_k\leq
        (2+h+C_0)\eta.
\end{aligned}
\]
Combining the initial and tail estimates gives
\[
        \limsup_{n\to\infty}
        \sup_{\substack{m\in\mathbb{Z}\\|m-n|\leq C_0}}
        \sum_{k=1}^{n+h}
        \min\{1,2^{k-m}\}a_k
        \leq
        (2+h+C_0)\eta.
\]
Letting \(\eta\downarrow0\) proves the deterministic statement.

On the probability-one event from Theorem~\ref{T:spine-lower-deviation}, 
apply the deterministic statement pathwise to \(a_k:=L_k(\theta)\), 
which is nonnegative and converges to zero for every \(0<\theta<\chi\).
\end{proof}

\subsection{Adaptive terminal cutoff}\label{SS:spine-terminal-cutoff}

For \(b>0\), \(h\in\mathbb{N}_0\), and \(n\geq1\), define
\[
        T_{n,h}^{>}(b)
        :=
        \sum_{|u|=n+h}
        \mu(I_u)
        \mathbf{1}_{\{\mu(I_u)>b/n\}}.
\]
Since \(\mathcal{D}_{n+h}([0,1])\) is a partition of \([0,1]\),
equivalently,
\[
        T_{n,h}^{>}(b)
        =
        \int_{[0,1]}
        \mathbf{1}_{\{\mu(I_{n+h}(t))>b/n\}}
        \,\mathrm{d}\mu(t).
\]

\begin{proposition}\label{P:terminal-positive-tail}
For every \(b>0\) and \(h\in\mathbb{N}_0\),
\[
        T_{n,h}^{>}(b)
        \longrightarrow0
        \qquad(n\to\infty)
\]
almost surely on \(\{Z_\infty>0\}\).
\end{proposition}

\begin{proof}
Fix an outcome in the probability-one event on which exact dimensionality holds, and suppose that \(Z_\infty>0\).  
Then, for \(\mu\)-almost every \(t\),
\[
        \lim_{m\to\infty}
        -\frac{1}{m}\log_2\mu(I_m(t))
        =
        \chi>0.
\]
Hence, for \(\mu\)-almost every \(t\), there exists \(m_0(t)<\infty\) such that
\(
\mu(I_m(t))\leq2^{-\chi m/2}
\)
for every \(m\geq m_0(t)\).  
Since \(n\,2^{-\chi(n+h)/2}\to0\), we have \(2^{-\chi(n+h)/2}<b/n\) for all sufficiently large \(n\).  
It follows that
\[
        \mathbf{1}_{\{\mu(I_{n+h}(t))>b/n\}}
        \longrightarrow0
\]
for \(\mu\)-almost every \(t\).

The integrand is bounded by \(1\), and \(\mu([0,1])=Z_\infty<\infty\).  
Dominated convergence applied to the integral representation of \(T_{n,h}^{>}(b)\) yields the assertion.
\end{proof}

For \(t\geq0\), define
\[
        H_Z(t)
        :=
        \mathbb{E}
        \left[
        Z_\infty\mathbf{1}_{\{Z_\infty>t\}}
        \right].
\]
The function \(H_Z\) is nonincreasing.  
Since \(Z_\infty\in L^1\) and \(\mathbb{E}Z_\infty=1\), 
dominated convergence gives \(H_Z(t)\to0\) as \(t\to\infty\).

For \(b>0\), \(h\in\mathbb{N}_0\), and \(n\geq1\), define
\[
        B_{n,h}^{\mathrm{term}}(b)
        :=
        \sum_{|u|=n+h}
        A_u
        H_Z\left(\frac{b}{nA_u}\right),
\]
where the summand is defined to be zero when \(A_u=0\).
Equivalently, for \(A_u>0\),
\[
        A_u
        H_Z\left(\frac{b}{nA_u}\right)
        =
        A_u
        \mathbb{E}
        \left[
        Z_\infty
        \mathbf{1}_{\{A_uZ_\infty>b/n\}}
        \right],
\]
where the expectation is taken with respect to an independent copy of \(Z_\infty\), with \(A_u\) held fixed.

\begin{proposition}\label{P:terminal-compensator}
For every \(b>0\) and \(h\in\mathbb{N}_0\),
\[
        B_{n,h}^{\mathrm{term}}(b)
        \longrightarrow0
        \qquad(n\to\infty)
\]
almost surely.
\end{proposition}

\begin{proof}
Fix \(0<\theta<\chi\), and put \(r:=n+h\).  
Split the sum according to whether \(V(u)\leq\theta r\) or \(V(u)>\theta r\).

Since \(H_Z(t)\leq\mathbb{E}Z_\infty=1\), the lower-deviation contribution satisfies
\[
\begin{aligned}
        \sum_{|u|=r}
        A_u
        H_Z\left(\frac{b}{nA_u}\right)
        \mathbf{1}_{\{V(u)\leq\theta r\}}
\leq\sum_{|u|=r}
        A_u\mathbf{1}_{\{V(u)\leq\theta r\}}
=L_r(\theta)
        \longrightarrow0
\end{aligned}
\]
almost surely, by Theorem~\ref{T:spine-lower-deviation}.

For the remaining contribution, terms with \(A_u=0\) vanish.  
If \(A_u>0\) and \(V(u)>\theta r\), then \(A_u<2^{-\theta r}\), and hence
\(
        \frac{b}{nA_u}
        >
        \frac{b2^{\theta r}}{n}.
\)
Since \(H_Z\) is nonincreasing,
\[
\begin{aligned}
        \sum_{|u|=r}
        A_u
        H_Z\left(\frac{b}{nA_u}\right)
        \mathbf{1}_{\{V(u)>\theta r\}}\leq
        H_Z\left(\frac{b2^{\theta r}}{n}\right)
        \sum_{|u|=r}
        A_u\mathbf{1}_{\{V(u)>\theta r\}}\leq
        H_Z\left(\frac{b2^{\theta r}}{n}\right)Z_r.
\end{aligned}
\]
Since \(r=n+h\),
\[
        \frac{b2^{\theta r}}{n}
        =
        b2^{\theta h}\frac{2^{\theta n}}{n}
        \longrightarrow\infty.
\]
It follows that \(H_Z(b2^{\theta r}/n)\to0\).
Since also \(Z_r\to Z_\infty<\infty\) almost surely, 
the typical contribution tends to zero almost surely.  
Combining the two parts proves the proposition.
\end{proof}

\begin{theorem}\label{T:terminal-adaptive-cutoff}
For every \(b>0\) and \(h\in\mathbb{N}_0\),
\[
        T_{n,h}^{>}(b)\longrightarrow0,
        \qquad
        B_{n,h}^{\mathrm{term}}(b)\longrightarrow0
        \qquad(n\to\infty)
\]
almost surely on \(\{Z_\infty>0\}\).
\end{theorem}

\begin{proof}
The first convergence is Proposition~\ref{P:terminal-positive-tail}, 
and the second is Proposition~\ref{P:terminal-compensator}.
\end{proof}

\subsection{Terminal descendants over planar grids}\label{SS:spine-terminal-descendants}

We first recall the classical real-valued Bernstein inequality
\cite[Theorem~2.8.4, p.~38]{VershyninHDP}.

\begin{lemma}[Real-valued Bernstein inequality]
\label{L:bernstein-real-valued}
Let \(Y_1,\ldots,Y_N\) be independent, mean-zero, real-valued random variables such that
\(|Y_k|\leq K\) almost surely for every \(1\leq k\leq N\).
Set
\(
\sigma^2=\sum_{k=1}^{N}\mathbb{E}Y_k^2.
\)
Then, for every \(t\geq0\),
\[
        \mathbb{P}
        \left(
        \left|
        \sum_{k=1}^{N}Y_k
        \right|
        \geq t
        \right)
        \leq
        2
        \exp
        \left(
        -\frac{t^2}{2(\sigma^2+Kt/3)}
        \right).
\]
\end{lemma}

We shall use the following conditional complex-valued extension.

\begin{lemma}[Conditional complex Bernstein inequality]
\label{L:bernstein-conditional-complex}
Let
\(
        \mathcal{G}\subset\mathcal{F}
\)
be a sub-\(\sigma\)-algebra, and let
\(
        X_1,\ldots,X_N
\)
be square-integrable complex-valued random variables which are conditionally independent given \(\mathcal{G}\).  
Suppose that
\[
        \mathbb{E}
        \left(
        X_k
        \,\middle|\,
        \mathcal{G}
        \right)
        =
        0
        \qquad\text{almost surely}
\]
for every \(1\leq k\leq N\).

Assume that there exist finite nonnegative \(\mathcal{G}\)-measurable random variables \(R\) and \(V\) such that
\[
        |X_k|\leq R
        \qquad\text{almost surely for every }1\leq k\leq N,
\]
and
\[
        \sum_{k=1}^{N}
        \mathbb{E}
        \left(
        |X_k|^2
        \,\middle|\,
        \mathcal{G}
        \right)
        \leq V
        \qquad\text{almost surely}.
\]
Then there exist absolute constants \(C,c>0\) such that, for every \(t>0\),
\[
        \mathbb{P}
        \left(
        \left|
        \sum_{k=1}^{N}X_k
        \right|
        \geq t
        \,\middle|\,
        \mathcal{G}
        \right)
        \leq
        C
        \exp
        \left(
        -\frac{ct^2}{V+Rt}
        \right)
        \qquad\text{almost surely}.
\]
\end{lemma}

\begin{proof}
The standard exponential-moment proof of Lemma \ref{L:bernstein-real-valued}, 
with conditional expectations in place of ordinary expectations, 
gives the following conditional real-valued version.  
If \(Y_1,\ldots,Y_N\) are conditionally independent and centered given \(\mathcal{G}\), with
\[
        |Y_k|\leq R,
        \qquad
        \sum_{k=1}^{N}
        \mathbb{E}
        \left(
        Y_k^2
        \,\middle|\,
        \mathcal{G}
        \right)
        \leq V,
\]
then, for every \(u\geq0\),
\[
        \mathbb{P}
        \left(
        \left|
        \sum_{k=1}^{N}Y_k
        \right|
        \geq u
        \,\middle|\,
        \mathcal{G}
        \right)
        \leq
        2\exp
        \left(
        -\frac{u^2}{2(V+Ru/3)}
        \right)
\]
almost surely.

The real-valued random variables
\(
\operatorname{Re}X_1,\ldots,\operatorname{Re}X_N
\)
are conditionally independent and centered given \(\mathcal{G}\), and the same is true of
\(
\operatorname{Im}X_1,\ldots,\operatorname{Im}X_N.
\)
Moreover,
\[
        |\operatorname{Re}X_k|,
        |\operatorname{Im}X_k|
        \leq R,
\]
and
\[
        \sum_{k=1}^{N}
        \mathbb{E}
        \left(
        (\operatorname{Re}X_k)^2
        \,\middle|\,
        \mathcal{G}
        \right)
        \leq V,
        \qquad
        \sum_{k=1}^{N}
        \mathbb{E}
        \left(
        (\operatorname{Im}X_k)^2
        \,\middle|\,
        \mathcal{G}
        \right)
        \leq V.
\]
Since
\[
        |z|\geq t
        \quad\Longrightarrow\quad
        |\operatorname{Re}z|\geq\frac{t}{2}
        \quad\text{or}\quad
        |\operatorname{Im}z|\geq\frac{t}{2},
\]
the conditional union bound gives
\[
        \mathbb{P}
        \left(
        \left|
        \sum_{k=1}^{N}X_k
        \right|
        \geq t
        \,\middle|\,
        \mathcal{G}
        \right)
        \leq
        4\exp
        \left(
        -\frac{(t/2)^2}
        {2(V+Rt/6)}
        \right)
        \leq
        4\exp
        \left(
        -\frac{t^2}{8(V+Rt)}
        \right).
\]
Thus the assertion holds with \(C=4\) and \(c=1/8\).
\end{proof}

\begin{theorem}[Terminal descendants over planar grids]
\label{T:terminal-descendants-planar}
Fix \(K_1\geq1\) and \(h\in\mathbb{N}_0\).  
Let
\[
        \Lambda_n
        \subset
        \left\{
        \xi\in\mathbb{R}^2:
        2^n\leq|\xi|<2^{n+1}
        \right\},
        \qquad n\geq1,
\]
be deterministic finite sets such that, for some \(C_\Lambda<\infty\),
\[
        \#\Lambda_n
        \leq
        C_\Lambda2^{2n}
\]
for every \(n\geq1\).  
For \(|u|=n+h\) and \(\xi\in\Lambda_n\), 
let \(\gamma_{u,n}(\xi)\) be an \(\mathcal{F}_{n+h}\)-measurable complex coefficient satisfying
\(|\gamma_{u,n}(\xi)|\leq1\) almost surely.  
Then
\[
        \sup_{\xi\in\Lambda_n}
        \left|
        \sum_{|u|=n+h}
        A_u
        \bigl(Z_\infty^{(u)}-1\bigr)
        \gamma_{u,n}(\xi)
        \right|
        \longrightarrow0
        \qquad(n\to\infty)
\]
almost surely on \(\mathcal{Z}_{K_1}\cap\{Z_\infty>0\}\).
\end{theorem}

\begin{proof}
Fix a rational number \(a>0\).  
Put \(r:=n+h\), and let \(\tau_n:=b/n\), where the rational number \(b>0\) will be chosen below.  
For \(x\geq0\), define
\[
        \beta_n^{\mathrm{term}}(x)
        :=
        \mathbb{E}
        \left[
        Z_\infty
        \mathbf{1}_{\{xZ_\infty\leq\tau_n\}}
        \right].
\]
For \(|u|=r\) and \(\xi\in\Lambda_n\), set
\[
\begin{aligned}
        Y_{u,n}^{\mathrm{term}}(\xi)
        &:=
        \left(
        A_uZ_\infty^{(u)}
        \mathbf{1}_{\{
        A_uZ_\infty^{(u)}\leq\tau_n
        \}}
        -
        A_u\beta_n^{\mathrm{term}}(A_u)
        \right)
        \gamma_{u,n}(\xi).
\end{aligned}
\]
Conditionally on \(\mathcal{F}_r\), the variables \(Y_{u,n}^{\mathrm{term}}(\xi)\), \(|u|=r\), are independent.  
Moreover, since \(A_u\) and \(\gamma_{u,n}(\xi)\) are \(\mathcal{F}_r\)-measurable and the variables
\(Z_\infty^{(u)}\), \(|u|=r\), 
are independent copies of \(Z_\infty\), independent of \(\mathcal{F}_r\),
\[
        \mathbb{E}
        \left[
        Y_{u,n}^{\mathrm{term}}(\xi)
        \mid\mathcal{F}_r
        \right]
        =
        0.
\]
Both the truncated term and its conditional expectation have absolute value at most \(\tau_n\).  
Hence
\[
        \left|
        Y_{u,n}^{\mathrm{term}}(\xi)
        \right|
        \leq
        2\tau_n
        =
        \frac{2b}{n}.
\]

For the conditional second moment, conditional centering and \(|\gamma_{u,n}(\xi)|\leq1\) give
\[
\begin{aligned}
        \mathbb{E}
        \left[
        \left|
        Y_{u,n}^{\mathrm{term}}(\xi)
        \right|^2
        \mid\mathcal{F}_r
        \right]
        &\leq
        A_u^2
        \mathbb{E}
        \left[
        \left(Z_\infty^{(u)}\right)^2
        \mathbf{1}_{\{
        A_uZ_\infty^{(u)}\leq\tau_n
        \}}
        \,\middle|\,
        \mathcal{F}_r
        \right]
\\
        &\leq
        \tau_nA_u
        \mathbb{E}
        \left[
        Z_\infty^{(u)}
        \mid\mathcal{F}_r
        \right]
\\
        &=
        \tau_nA_u.
\end{aligned}
\]
Consequently,
\[
        \sum_{|u|=r}
        \mathbb{E}
        \left[
        \left|
        Y_{u,n}^{\mathrm{term}}(\xi)
        \right|^2
        \mid\mathcal{F}_r
        \right]
        \leq
        \tau_nZ_r
        =
        \frac{b}{n}Z_r.
\]

For fixed \(\xi\in\Lambda_n\), let
\[
        E_{n,\xi}
        :=
        \left\{
        \left|
        \sum_{|u|=r}
        Y_{u,n}^{\mathrm{term}}(\xi)
        \right|
        >
        \frac{a}{3}
        \right\}.
\]
Applying Lemma~\ref{L:bernstein-conditional-complex} conditionally on
\(\mathcal{F}_r\), with
\[
        R=\frac{2b}{n},
        \qquad
        V=\frac{b}{n}Z_r,
        \qquad
        t=\frac{a}{3},
\]
we obtain, on \(\{Z_r\leq K_1\}\),
\[
        \mathbb{P}
        \left(
        E_{n,\xi}
        \mid\mathcal{F}_r
        \right)
        \leq
        C\exp
        \left(
        -c_{a,K_1}\frac{n}{b}
        \right),
\]
where \(c_{a,K_1}>0\) is independent of \(n\), \(\xi\), and \(b\).
Since
\(\mathcal{Z}_{K_1}\subseteq\{Z_r\leq K_1\}\) and
\(\{Z_r\leq K_1\}\in\mathcal{F}_r\),
\[
        \mathbb{P}
        \left(
        E_{n,\xi}
        \cap\mathcal{Z}_{K_1}
        \right)
        \leq
        \mathbb{P}
        \left(
        E_{n,\xi}
        \cap\{Z_r\leq K_1\}
        \right)
        =
        \mathbb{E}
        \left[
        \mathbf{1}_{\{Z_r\leq K_1\}}
        \mathbb{P}
        \left(
        E_{n,\xi}
        \mid\mathcal{F}_r
        \right)
        \right]
        \leq
        C\exp
        \left(
        -c_{a,K_1}\frac{n}{b}
        \right).
\]
Taking a union bound over \(\Lambda_n\), we obtain
\[
\begin{aligned}
        \mathbb{P}
        \left(
        \left\{
        \sup_{\xi\in\Lambda_n}
        \left|
        \sum_{|u|=r}
        Y_{u,n}^{\mathrm{term}}(\xi)
        \right|
        >
        \frac{a}{3}
        \right\}
        \cap\mathcal{Z}_{K_1}
        \right)\leq
        CC_\Lambda
        2^{2n}
        \exp
        \left(
        -c_{a,K_1}\frac{n}{b}
        \right).
\end{aligned}
\]
Choose \(b>0\) rational and sufficiently small that \(c_{a,K_1}/b>2\log2\).
The preceding probabilities are then summable in \(n\).  
By the Borel--Cantelli lemma, almost surely on \(\mathcal{Z}_{K_1}\),
\[
        \sup_{\xi\in\Lambda_n}
        \left|
        \sum_{|u|=r}
        Y_{u,n}^{\mathrm{term}}(\xi)
        \right|
        \leq
        \frac{a}{3}
\]
for all sufficiently large \(n\).

It remains to compare the original terminal descendant sum with the centered truncated sum.  
Since \(\mathbb{E}Z_\infty=1\), uniformly in \(\xi\in\Lambda_n\),
\[
\begin{aligned}
        &\left|
        \sum_{|u|=r}
        A_u
        \bigl(Z_\infty^{(u)}-1\bigr)
        \gamma_{u,n}(\xi)
        -
        \sum_{|u|=r}
        Y_{u,n}^{\mathrm{term}}(\xi)
        \right|
\\
        &\qquad\leq
        \sum_{|u|=r}
        A_uZ_\infty^{(u)}
        \mathbf{1}_{\{
        A_uZ_\infty^{(u)}>b/n
        \}}
        +
        \sum_{|u|=r}
        A_u
        \mathbb{E}
        \left[
        Z_\infty
        \mathbf{1}_{\{
        A_uZ_\infty>b/n
        \}}
        \right]
\\
        &\qquad=
        T_{n,h}^{>}(b)
        +
        B_{n,h}^{\mathrm{term}}(b),
\end{aligned}
\]
where, in the second sum, the expectation is taken with respect to an independent copy of \(Z_\infty\), 
with \(A_u\) held fixed.
The last equality uses \(\mu(I_u)=A_uZ_\infty^{(u)}\).

By Theorem~\ref{T:terminal-adaptive-cutoff}, the right-hand side tends to zero almost surely on \(\{Z_\infty>0\}\).  
Therefore, on \(\mathcal{Z}_{K_1}\cap\{Z_\infty>0\}\), 
the original terminal descendant supremum is eventually at most \(a\).  
Intersecting over all rational \(a>0\) proves the theorem.
\end{proof}

\subsection{The abstract stopped skeleton theorem}
\label{SS:spine-abstract-skeleton}

We next establish the abstract martingale criterion used in both the interval and fixed-arc proofs.  
For integers \(k\) and \(m\), set \(\omega_{k,m}:=\min\{1,2^{k-m}\}\).

The following complex-valued form follows from Freedman's inequality
\cite{Freedman1975} by applying the real-valued estimate to the real and imaginary parts.

\begin{lemma}\label{L:spine-complex-freedman}
Let \(N\in\mathbb{N}\), and let \((M_j,\mathcal{G}_j)_{j=0}^{N}\) be a complex-valued martingale 
with \(M_0=0\) and differences \(\Delta_j:=M_j-M_{j-1}\), \(1\leq j\leq N\).
Assume that there exist deterministic constants \(R,V>0\) such that,
almost surely,
\[
        |\Delta_j|\leq R
        \quad(1\leq j\leq N),
        \qquad
        \sum_{j=1}^{N}
        \mathbb{E}
        \left[
        |\Delta_j|^2\mid\mathcal{G}_{j-1}
        \right]
        \leq V.
\]
Then, for every \(t>0\),
\[
        \mathbb{P}
        \left(
        \max_{0\leq j\leq N}|M_j|>t
        \right)
        \leq
        4\exp
        \left(
        -\frac{t^2}{8(V+Rt)}
        \right).
\]
\end{lemma}

\begin{remark}
Indeed, \(|M_j|>t\) implies \(|\operatorname{Re}M_j|>t/\sqrt{2}\) or \(|\operatorname{Im}M_j|>t/\sqrt{2}\); 
applying the real-valued Freedman inequality to the two coordinate martingales and using a union bound yields the stated estimate, after weakening the constants.
\end{remark}

Fix \(h\in\mathbb{N}_0\), \(C_0\geq0\), \(M_0\in\mathbb{N}_0\), and \(0<C_\Gamma<\infty\).  
Let
\[
        \Lambda_n
        \subset
        \left\{
        \xi\in\mathbb{R}^2:
        2^n\leq|\xi|<2^{n+1}
        \right\},
        \qquad n\geq1,
\]
be deterministic nonempty finite sets satisfying \(\#\Lambda_n\leq C_\Lambda2^{2n}\) for some \(C_\Lambda<\infty\).  
For each \(n\geq1\) and \(\xi\in\Lambda_n\), let \(\mathcal{B}_{\xi,n}\) be a finite deterministic set such that
\(\#\mathcal{B}_{\xi,n}\leq M_0\).
For every \(d\in\mathcal{B}_{\xi,n}\), let \(m_{\xi,d}\in\mathbb{Z}\) satisfy \(|m_{\xi,d}-n|\leq C_0\).
For \(1\leq k\leq n+h\) and \(J\in\mathcal{D}_k([0,1])\), 
let \(\Gamma_J(\xi,d)\) be a complex \(\mathcal{F}_{k-1}\)-measurable coefficient satisfying, 
almost surely,
\[
        |\Gamma_J(\xi,d)|
        \leq
        C_\Gamma A_{J^{-}}\omega_{k,m_{\xi,d}},
        \qquad
        |\Gamma_J(\xi,d)|W_J
        \leq
        C_\Gamma A_J\omega_{k,m_{\xi,d}}.
\]
Define
\[
        \mathcal{S}_n(\xi)
        :=
        \sum_{d\in\mathcal{B}_{\xi,n}}
        \sum_{k=1}^{n+h}
        \sum_{J\in\mathcal{D}_k([0,1])}
        \Gamma_J(\xi,d)(W_J-1).
\]

\begin{theorem}[Abstract fixed-cutoff stopped skeleton]
\label{T:abstract-stopped-skeleton}
For every \(K_1\geq1\),
\[
        \sup_{\xi\in\Lambda_n}
        |\mathcal{S}_n(\xi)|
        \longrightarrow0
        \qquad(n\to\infty)
\]
almost surely on \(\mathcal{Z}_{K_1}\).
\end{theorem}

\begin{proof}
It is enough to prove that, for every rational \(a>0\),
\[
        \limsup_{n\to\infty}
        \sup_{\xi\in\Lambda_n}
        |\mathcal{S}_n(\xi)|
        \leq a
\]
almost surely on \(\mathcal{Z}_{K_1}\).

Fix such an \(a>0\), and fix a rational \(0<\theta<\chi\).  
Put \(\tau_n:=b/n\), where the rational number \(b>0\) will be chosen below.  
We divide the proof into three parts.

\boldparagraph{Positive large jumps.}

For \(\xi\in\Lambda_n\) and \(d\in\mathcal{B}_{\xi,n}\), define
\[
        R_n^{\mathrm{jump}}(\xi,d)
        :=
        \sum_{k=1}^{n+h}
        \sum_{J\in\mathcal{D}_k([0,1])}
        |\Gamma_J(\xi,d)|W_J
        \mathbf{1}_{\{
        |\Gamma_J(\xi,d)|W_J>\tau_n
        \}}.
\]
The child-mass coefficient bound implies
\[
\begin{aligned}
        R_n^{\mathrm{jump}}(\xi,d)
        &\leq
        C_\Gamma
        \sum_{k=1}^{n+h}
        \omega_{k,m_{\xi,d}}
        \sum_{J\in\mathcal{D}_k([0,1])}
        A_J
        \mathbf{1}_{\{
        A_J\omega_{k,m_{\xi,d}}>
        b/(C_\Gamma n)
        \}}.
\end{aligned}
\]

We claim that, for all sufficiently large \(n\), uniformly over
\(1\leq k\leq n+h\) and \(m\in\mathbb{Z}\) with \(|m-n|\leq C_0\),
\[
        \left\{
        A_J\omega_{k,m}>
        \frac{b}{C_\Gamma n}
        \right\}
        \subseteq
        \{V(J)\leq\theta k\}.
\]
Indeed, the event on the left implies \(A_J>0\) and
\[
        V(J)
        <
        \log_2\left(\frac{C_\Gamma n}{b}\right)
        +
        \log_2\omega_{k,m}.
\]
Let \(\vartheta:=\min\{\theta,1\}\).  
If \(k\leq m\), then \(\log_2\omega_{k,m}=k-m\), and
\[
        V(J)-\theta k
        <
        \log_2\left(\frac{C_\Gamma n}{b}\right)
        -
        \vartheta(m-1).
\]
Since \(m\geq n-C_0\), the right-hand side tends to \(-\infty\), uniformly in the allowed \(k\) and \(m\).  
If \(k>m\), then \(\omega_{k,m}=1\), while \(k>m\geq n-C_0\); hence
\[
        V(J)
        <
        \log_2\left(\frac{C_\Gamma n}{b}\right)
        \leq
        \theta k
\]
for all sufficiently large \(n\).  
This proves the claim.

Consequently, for all sufficiently large \(n\),
\[
        R_n^{\mathrm{jump}}(\xi,d)
        \leq
        C_\Gamma
        \sum_{k=1}^{n+h}
        \omega_{k,m_{\xi,d}}L_k(\theta).
\]
Lemma~\ref{L:spine-shifted-convolution} and \(\#\mathcal{B}_{\xi,n}\leq M_0\) therefore give
\[
        \sup_{\xi\in\Lambda_n}
        \sum_{d\in\mathcal{B}_{\xi,n}}
        R_n^{\mathrm{jump}}(\xi,d)
        \longrightarrow0
\]
almost surely.

\boldparagraph{Predictable compensators.}

For \(t\geq0\), define
\[
        H_W(t)
        :=
        \mathbb{E}
        \left[
        W\mathbf{1}_{\{W>t\}}
        \right].
\]
The function \(H_W\) is nonincreasing, and \(H_W(t)\to0\) as \(t\to\infty\).  
Define
\[
        R_n^{\mathrm{comp}}(\xi,d)
        :=
        \sum_{k=1}^{n+h}
        \sum_{J\in\mathcal{D}_k([0,1])}
        |\Gamma_J(\xi,d)|
        H_W\left(
        \frac{\tau_n}{|\Gamma_J(\xi,d)|}
        \right),
\]
where the summand is defined to be zero when \(\Gamma_J(\xi,d)=0\).

The contribution from \(k=1\) satisfies, uniformly in \(\xi\) and \(d\),
\[
\begin{aligned}
        \sum_{J\in\mathcal{D}_1([0,1])}
        |\Gamma_J(\xi,d)|
        H_W\left(
        \frac{\tau_n}{|\Gamma_J(\xi,d)|}
        \right)
        &\leq
        \sum_{J\in\mathcal{D}_1([0,1])}
        |\Gamma_J(\xi,d)|
\\
        &\leq
        2C_\Gamma\omega_{1,m_{\xi,d}}
\\
        &\leq
        C2^{-n}.
\end{aligned}
\]

Now let \(k\geq2\), write \(I:=J^{-}\), and first consider the region \(V(I)>\theta(k-1)\).  
Then \(A_I<2^{-\theta(k-1)}\).  
If \(\Gamma_J(\xi,d)\neq0\), the parent-mass coefficient bound gives
\[
        \frac{\tau_n}{|\Gamma_J(\xi,d)|}
        \geq
        \frac{b}{C_\Gamma n}
        2^{\theta(k-1)}
        \omega_{k,m_{\xi,d}}^{-1}.
\]
Let \(\vartheta:=\min\{\theta,1\}\).  
If \(k\leq m_{\xi,d}\), then
\[
        2^{\theta(k-1)}
        \omega_{k,m_{\xi,d}}^{-1}
        =
        2^{\theta(k-1)+m_{\xi,d}-k}
        \geq
        2^{\vartheta(m_{\xi,d}-1)}.
\]
The same lower bound holds when \(k>m_{\xi,d}\).  
Since \(m_{\xi,d}\geq n-C_0\), there exists \(c_0=c_0(\theta)>0\), independent of \(b\), 
such that, for all sufficiently large \(n\),
\[
        \frac{\tau_n}{|\Gamma_J(\xi,d)|}
        \geq
        2^{c_0n}
\]
uniformly over all allowed \(\xi,d,k,J\).

It follows that the typical-parent contribution is bounded by
\[
C H_W(2^{c_0n}) \sum_{k=1}^{n+h} \omega_{k,m_{\xi,d}}Z_{k-1}.
\]
On \(\mathcal{Z}_{K_1}\), this is at most
\(CK_1 H_W(2^{c_0n}) \sum_{k=1}^{n+h} \omega_{k,m_{\xi,d}}\),
which tends to zero uniformly in \(\xi,d\), since
\[
        \sup_{\substack{m\in\mathbb{Z}\\|m-n|\leq C_0}}
        \sum_{k=1}^{n+h}
        \omega_{k,m}
        \leq
        2+h+C_0
\]
for all sufficiently large \(n\).

On the region \(V(I)\leq\theta(k-1)\), use only \(H_W\leq1\).  
Since each parent has two children,
\[
\begin{aligned}
        &\sum_{k=2}^{n+h}
        \sum_{J\in\mathcal{D}_k([0,1])}
        |\Gamma_J(\xi,d)|
        \mathbf{1}_{\{V(J^{-})\leq\theta(k-1)\}}
\\
        &\qquad\leq
        2C_\Gamma
        \sum_{k=2}^{n+h}
        \omega_{k,m_{\xi,d}}L_{k-1}(\theta)
\\
        &\qquad\leq
        4C_\Gamma
        \sum_{q=1}^{n+h}
        \omega_{q,m_{\xi,d}}L_q(\theta),
\end{aligned}
\]
where we used \(\omega_{q+1,m}\leq2\omega_{q,m}\).  
By Lemma~\ref{L:spine-shifted-convolution}, this tends to zero uniformly in \(\xi,d\).  
Hence
\[
        \sup_{\xi\in\Lambda_n}
        \sum_{d\in\mathcal{B}_{\xi,n}}
        R_n^{\mathrm{comp}}(\xi,d)
        \longrightarrow0
\]
almost surely on \(\mathcal{Z}_{K_1}\).

\boldparagraph{Centered capped martingales.}

Let
\[
        \tau_{K_1}
        :=
        \inf\{\ell\geq0:Z_\ell>K_1\},
\]
with the convention \(\inf\varnothing=\infty\).  
For \(x\geq0\), define
\[
        \beta_n^{\mathrm{cap}}(x)
        :=
        \mathbb{E}
        \left[
        W\mathbf{1}_{\{xW\leq\tau_n\}}
        \right].
\]
For \(J\in\mathcal{D}_k([0,1])\), set
\[
\begin{aligned}
        X_{J,n}^{\mathrm{cap}}(\xi,d)
        &:=
        \mathbf{1}_{\{k-1<\tau_{K_1}\}}
        \Gamma_J(\xi,d)
        \left(
        W_J
        \mathbf{1}_{\{
        |\Gamma_J(\xi,d)|W_J\leq\tau_n
        \}}
        -
        \beta_n^{\mathrm{cap}}(|\Gamma_J(\xi,d)|)
        \right).
\end{aligned}
\]
Since \(\{k-1<\tau_{K_1}\}\) and \(\Gamma_J(\xi,d)\) are \(\mathcal{F}_{k-1}\)-measurable 
and \(W_J\) is independent of \(\mathcal{F}_{k-1}\),
\[
        \mathbb{E}
        \left[
        X_{J,n}^{\mathrm{cap}}(\xi,d)
        \mid\mathcal{F}_{k-1}
        \right]
        =0.
\]
Moreover,
\[
        |X_{J,n}^{\mathrm{cap}}(\xi,d)|
        \leq
        2\tau_n
        =
        \frac{2b}{n}.
\]
Conditional centering gives
\[
\begin{aligned}
        \mathbb{E}
        \left[
        |X_{J,n}^{\mathrm{cap}}(\xi,d)|^2
        \mid\mathcal{F}_{k-1}
        \right]
\leq
        \tau_n
        \mathbf{1}_{\{k-1<\tau_{K_1}\}}
        |\Gamma_J(\xi,d)|.
\end{aligned}
\]
Also,
\[
        \sum_{J\in\mathcal{D}_k([0,1])}
        |\Gamma_J(\xi,d)|
        \leq
        2C_\Gamma
        \omega_{k,m_{\xi,d}}Z_{k-1}.
\]
Therefore the stopped predictable quadratic variation is bounded by
\[
\begin{aligned}
        \tau_n
        \sum_{k=1}^{n+h}
        \mathbf{1}_{\{k-1<\tau_{K_1}\}}
        \sum_{J\in\mathcal{D}_k([0,1])}
        |\Gamma_J(\xi,d)|\leq
        C(C_0,h,C_\Gamma,K_1)\frac{b}{n}.
\end{aligned}
\]

For fixed \(\xi,d\), order the pairs \((k,J)\) first by generation and then in a deterministic order within each generation, 
and let the filtration reveal the corresponding weights in this order.
Independence of the weights within each generation ensures that the ordered variables \(X_{J,n}^{\mathrm{cap}}(\xi,d)\) remain martingale differences with respect to this filtration.  
Hence Lemma~\ref{L:spine-complex-freedman} gives, for every \(a_d>0\),
\[
        \mathbb{P}
        \left(
        \left|
        \sum_{k=1}^{n+h}
        \sum_{J\in\mathcal{D}_k([0,1])}
        X_{J,n}^{\mathrm{cap}}(\xi,d)
        \right|
        >
        a_d
        \right)
        \leq
        4\exp\left(-c\frac{n}{b}\right),
\]
where
\(c=c(a_d,C_0,h,C_\Gamma,K_1)>0\) is independent of \(n,\xi,d,b\).

If \(M_0=0\), then \(\mathcal{S}_n\equiv0\).  
Otherwise, set \(a_d:=a/(3M_0)\). 
A union bound over \(\xi\in\Lambda_n\) and \(d\in\mathcal{B}_{\xi,n}\) gives
\[
\begin{aligned}
        \mathbb{P}
        \left(
        \sup_{\xi\in\Lambda_n}
        \sum_{d\in\mathcal{B}_{\xi,n}}
        \left|
        \sum_{k=1}^{n+h}
        \sum_{J\in\mathcal{D}_k([0,1])}
        X_{J,n}^{\mathrm{cap}}(\xi,d)
        \right|
        >
        \frac{a}{3}
        \right)
\leq
        4C_\Lambda M_0
        2^{2n}
        \exp\left(-c\frac{n}{b}\right).
\end{aligned}
\]
Choose \(b>0\) rational and sufficiently small that \(c/b>2\log 2\).  
The right-hand side is then summable in \(n\).
By the Borel--Cantelli lemma, almost surely,
\[
        \sup_{\xi\in\Lambda_n}
        \sum_{d\in\mathcal{B}_{\xi,n}}
        \left|
        \sum_{k=1}^{n+h}
        \sum_{J\in\mathcal{D}_k([0,1])}
        X_{J,n}^{\mathrm{cap}}(\xi,d)
        \right|
        \leq
        \frac{a}{3}
\]
for all sufficiently large \(n\).

On \(\mathcal{Z}_{K_1}\), the stopping never occurs.  
Since \(\mathbb{E}W=1\), for every \(\xi,d\),
\[
\begin{aligned}
        \left|
        \sum_{k=1}^{n+h}
        \sum_{J\in\mathcal{D}_k([0,1])}
        \Gamma_J(\xi,d)(W_J-1)
        -
        \sum_{k=1}^{n+h}
        \sum_{J\in\mathcal{D}_k([0,1])}
        X_{J,n}^{\mathrm{cap}}(\xi,d)
        \right|\leq
        R_n^{\mathrm{jump}}(\xi,d)
        +
        R_n^{\mathrm{comp}}(\xi,d).
\end{aligned}
\]
The positive large-jump and predictable-compensator estimates imply that, 
after summing over bands, the right-hand side tends to zero uniformly in \(\xi\in\Lambda_n\) on \(\mathcal{Z}_{K_1}\).
Consequently,
\[
        \limsup_{n\to\infty}
        \sup_{\xi\in\Lambda_n}
        |\mathcal{S}_n(\xi)|
        \leq
        \frac{a}{3}
        \leq
        a
\]
almost surely on \(\mathcal{Z}_{K_1}\).  
Intersecting over all rational \(a>0\) proves the theorem.
\end{proof}

\begin{remark}\label{R:spine-skeleton-threshold-form}
In applications of Theorem~\ref{T:abstract-stopped-skeleton}, after fixing rational \(a>0\), 
rational \(0<\theta<\chi\), and \(C_\Lambda,K_1,h,M_0,C_0,C_\Gamma\), 
choose a rational cap \(b>0\) small enough that the capped Freedman bound is summable over the planar grid. 
The large-jump and compensator terms vanish by the shifted convolution estimate and
\(
\mathbb{E}\!\left[W\mathbf{1}_{\{W>t\}}\right]\to0.
\)
Intersecting over rational \(a\) yields convergence; 
no uniform choice of \(b\) is required.
\end{remark}

\subsection{The interval Rajchman theorem}\label{SS:spine-interval-assembly}

\begin{proof}[Proof of Theorem~\ref{T:main-interval-rajchman}]
For \(n\geq1\), set
\[
        S_n^{\mathrm{int}}
        :=
        \sup_{2^n\leq|\xi|<2^{n+1}}
        |\widehat{\mu}(\xi)|.
\]
It is enough to prove that
\[
        S_n^{\mathrm{int}}
        \longrightarrow0
        \qquad(n\to\infty)
\]
almost surely on \(\{Z_\infty>0\}\).

Fix a rational tolerance \(\varepsilon>0\) and \(K,K_1\in\mathbb{N}\). 
Work on the localized event
\[
        \Omega_K
        \cap
        \mathcal{Z}_{K_1}
        \cap
        \{Z_\infty>0\}.
\]

\boldparagraph{Gridding.}

On \(\Omega_K\),
\[
\begin{aligned}
        |\widehat{\mu}(\xi)-\widehat{\mu}(\eta)|
        \leq
        \int_0^1
        \left|
        \mathrm{e}^{-2\pi i\xi t}
        -
        \mathrm{e}^{-2\pi i\eta t}
        \right|
        \,\mathrm{d}\mu(t)
\leq
        2\pi|\xi-\eta|Z_\infty
\leq
        2\pi K|\xi-\eta|.
\end{aligned}
\]
Let \(\delta:=\varepsilon/(64\pi K)\).
For each \(n\geq1\), choose a deterministic \(\delta\)-net
\[
        \Lambda_n^{\mathrm{int}}(K,\varepsilon)
        \subset
        \left\{
        \xi\in\mathbb{R}:
        2^n\leq|\xi|<2^{n+1}
        \right\}
\]
with
\[
        \#\Lambda_n^{\mathrm{int}}(K,\varepsilon)
        \leq
        C_{K,\varepsilon}2^n.
\]
Then, on \(\Omega_K\),
\[
        S_n^{\mathrm{int}}>\varepsilon
        \quad\Longrightarrow\quad
        \exists\,\xi
        \in
        \Lambda_n^{\mathrm{int}}(K,\varepsilon)
        \ \text{such that}\
        |\widehat{\mu}(\xi)|
        >
        \frac{31}{32}\varepsilon.
\]
We regard \(\Lambda_n^{\mathrm{int}}(K,\varepsilon)\) as a subset of \(\mathbb{R}^2\) by embedding \(\mathbb{R}\) as the first coordinate axis.  
The embedded grid lies in the corresponding planar annulus and satisfies
\[
        \#\Lambda_n^{\mathrm{int}}(K,\varepsilon)
        \leq
        C_{K,\varepsilon}2^{2n}.
\]

\boldparagraph{Limiting atomization.}

Let \(h\in\mathbb{N}_0\), to be chosen in the conclusion, and put \(r:=n+h\). 
Fix once and for all a deterministic point \(t_u\in I_u\) for every finite word \(u\).
If \(|u|=r\), \(t\in I_u\), and \(2^n\leq|\xi|<2^{n+1}\), then
\[
        |\xi||t-t_u|
        \leq
        2^{n+1}2^{-n-h}
        =
        2^{1-h}.
\]
Consequently, on \(\Omega_K\),
\[
        \left|
        \widehat{\mu}(\xi)
        -
        \sum_{|u|=n+h}
        \mu(I_u)\mathrm{e}^{-2\pi i\xi t_u}
        \right|
        \leq
        CK2^{-h}
\]
uniformly over \(2^n\leq|\xi|<2^{n+1}\).

Using \(\mu(I_u)=A_uZ_\infty^{(u)}\), we have
\[
\begin{aligned}
        \sum_{|u|=n+h}
        \mu(I_u)\mathrm{e}^{-2\pi i\xi t_u}
        =
        \sum_{|u|=n+h}
        A_u\mathrm{e}^{-2\pi i\xi t_u}+
        \sum_{|u|=n+h}
        A_u
        \bigl(Z_\infty^{(u)}-1\bigr)
        \mathrm{e}^{-2\pi i\xi t_u}.
\end{aligned}
\]
The coefficients
\(\gamma_{u,n}(\xi):=\mathrm{e}^{-2\pi i\xi t_u}\) are deterministic and satisfy \(|\gamma_{u,n}(\xi)|=1\).
Hence Theorem~\ref{T:terminal-descendants-planar} gives, for fixed \(K,K_1,\varepsilon,h\),
\[
        \sup_{\xi\in\Lambda_n^{\mathrm{int}}(K,\varepsilon)}
        \left|
        \sum_{|u|=n+h}
        A_u
        \bigl(Z_\infty^{(u)}-1\bigr)
        \mathrm{e}^{-2\pi i\xi t_u}
        \right|
        \longrightarrow0
\]
almost surely on \(\mathcal{Z}_{K_1}\cap\{Z_\infty>0\}\).

\boldparagraph{Level atomization.}

The level-\((n+h)\) approximation is
\[
        \mathrm{d}\mu_{n+h}(t)
        =
        \sum_{|u|=n+h}
        2^{n+h}A_u
        \mathbf{1}_{I_u}(t)
        \,\mathrm{d}t.
\]
On \(\mathcal{Z}_{K_1}\), \(\mu_{n+h}([0,1])=Z_{n+h}\leq K_1\).
The same Lipschitz estimate gives
\[
        \left|
        \sum_{|u|=n+h}
        A_u\mathrm{e}^{-2\pi i\xi t_u}
        -
        \int_0^1
        \mathrm{e}^{-2\pi i\xi t}
        \,\mathrm{d}\mu_{n+h}(t)
        \right|
        \leq
        CK_1 2^{-h}
\]
uniformly over \(2^n\leq|\xi|<2^{n+1}\).

\boldparagraph{Finite-level martingale identity.}

For \(J\in\mathcal{D}_k([0,1])\), define
\[
        c_J(\xi,k-1)
        :=
        2^{k-1}
        \int_J
        \mathrm{e}^{-2\pi i\xi t}
        \,\mathrm{d}t.
\]
Since \(\mu_0\) is Lebesgue measure and
\[
        \mathrm{d}(\mu_k-\mu_{k-1})(t)
        =
        \sum_{J\in\mathcal{D}_k([0,1])}
        2^{k-1}
        A_{J^{-}}
        (W_J-1)
        \mathbf{1}_J(t)
        \,\mathrm{d}t,
\]
we obtain
\[
\begin{aligned}
        \int_0^1
        \mathrm{e}^{-2\pi i\xi t}
        \,\mathrm{d}\mu_{n+h}(t)
        =
        \int_0^1
        \mathrm{e}^{-2\pi i\xi t}
        \,\mathrm{d}t+
        \sum_{k=1}^{n+h}
        \sum_{J\in\mathcal{D}_k([0,1])}
        A_{J^{-}}
        c_J(\xi,k-1)
        (W_J-1).
\end{aligned}
\]
The deterministic Lebesgue term satisfies
\[
        \left|
        \int_0^1
        \mathrm{e}^{-2\pi i\xi t}
        \,\mathrm{d}t
        \right|
        \leq
        C2^{-n}.
\]
Moreover, whenever \(2^n\leq|\xi|<2^{n+1}\),
\[
\begin{aligned}
        |c_J(\xi,k-1)|
        \leq
        C\min
        \left\{
        2^{k-1}|J|,
        2^{k-1}|\xi|^{-1}
        \right\}\leq
        C\min\{1,2^{k-n}\}.
\end{aligned}
\]

Set
\[
        \Gamma_J(\xi)
        :=
        A_{J^{-}}c_J(\xi,k-1),
        \qquad
        J\in\mathcal{D}_k([0,1]).
\]
The coefficient \(\Gamma_J(\xi)\) is \(\mathcal{F}_{k-1}\)-measurable, and
\[
        |\Gamma_J(\xi)|
        \leq
        C A_{J^{-}}\min\{1,2^{k-n}\}.
\]
Since \(A_{J^{-}}W_J=2A_J\), after enlarging \(C\) if necessary,
\[
        |\Gamma_J(\xi)|W_J
        \leq
        C A_J\min\{1,2^{k-n}\}.
\]
Thus Theorem~\ref{T:abstract-stopped-skeleton} applies with
\[
        \mathcal{B}_{\xi,n}:=\{d_0\},
        \qquad
        m_{\xi,d_0}:=n,
        \qquad
        M_0:=1,
        \qquad
        C_0:=0,
\]
after identifying \(\Gamma_J(\xi,d_0)\) with \(\Gamma_J(\xi)\). 
Therefore, for fixed \(K,K_1,\varepsilon,h\),
\[
        \sup_{\xi\in\Lambda_n^{\mathrm{int}}(K,\varepsilon)}
        \left|
        \sum_{k=1}^{n+h}
        \sum_{J\in\mathcal{D}_k([0,1])}
        A_{J^{-}}
        c_J(\xi,k-1)
        (W_J-1)
        \right|
        \longrightarrow0
\]
almost surely on \(\mathcal{Z}_{K_1}\).

\boldparagraph{Conclusion.}

On the simultaneous probability-one event fixed at the beginning of this section, and on
\(
\Omega_K\cap\mathcal{Z}_{K_1},
\)
the following bound holds for every \(n\geq1\):
\[
\begin{aligned}
        \sup_{\xi\in\Lambda_n^{\mathrm{int}}(K,\varepsilon)}
        |\widehat{\mu}(\xi)|
        &\leq
        CK2^{-h}
\\
        &\quad+
        \sup_{\xi\in\Lambda_n^{\mathrm{int}}(K,\varepsilon)}
        \left|
        \sum_{|u|=n+h}
        A_u\bigl(Z_\infty^{(u)}-1\bigr)
        \mathrm{e}^{-2\pi i\xi t_u}
        \right|
\\
        &\quad+
        CK_1 2^{-h}
        +
        C2^{-n}
\\
        &\quad+
        \sup_{\xi\in\Lambda_n^{\mathrm{int}}(K,\varepsilon)}
        \left|
        \sum_{k=1}^{n+h}
        \sum_{J\in\mathcal{D}_k([0,1])}
        A_{J^{-}}c_J(\xi,k-1)(W_J-1)
        \right|.
\end{aligned}
\]
Choose \(h\in\mathbb{N}_0\) so large that
\[
        CK2^{-h}<\frac{\varepsilon}{100},
        \qquad
        CK_1 2^{-h}<\frac{\varepsilon}{100}.
\]
With this deterministic \(h\) fixed, the terminal descendant supremum converges to zero almost surely on
\(\mathcal{Z}_{K_1}\cap\{Z_\infty>0\}\), 
the stopped-skeleton supremum converges to zero almost surely on \(\mathcal{Z}_{K_1}\), 
and \(C2^{-n}\to0\) deterministically.  
Hence, for all sufficiently large \(n\),
\(
   \sup_{\xi\in\Lambda_n^{\mathrm{int}}(K,\varepsilon)}
        |\widehat{\mu}(\xi)|
        <
        \frac{\varepsilon}{4}
\)
on \(\Omega_K\cap\mathcal{Z}_{K_1}\cap\{Z_\infty>0\}\).
The gridding implication therefore rules out \(S_n^{\mathrm{int}}>\varepsilon\) 
for all sufficiently large \(n\) on this localized event.

Since \(Z_n\to Z_\infty<\infty\) almost surely, the sequence \((Z_n)\) is almost surely bounded.  
Define
\[
        \mathcal{L}
        :=
        \bigcup_{K=1}^{\infty}
        \bigcup_{K_1=1}^{\infty}
        \Omega_K\cap\mathcal{Z}_{K_1}.
\]
Then \(\mathbb{P}(\mathcal{L})=1\).
For each \((K,K_1,\varepsilon)\in\mathbb{N}^2\times\mathbb{Q}_{>0}\),
let \(E_{K,K_1,\varepsilon}\) be a probability-one event on which the preceding eventual estimate holds on
\[
        \Omega_K
        \cap
        \mathcal{Z}_{K_1}
        \cap
        \{Z_\infty>0\}.
\]
Since the parameter set is countable, the event
\[
        E
        :=
        \mathcal{L}
        \cap
        \bigcap_{K=1}^{\infty}
        \bigcap_{K_1=1}^{\infty}
        \bigcap_{\varepsilon\in\mathbb{Q}_{>0}}
        E_{K,K_1,\varepsilon}
\]
has probability one. 
If \(\omega\in E\cap\{Z_\infty>0\}\), then \(\omega\in\mathcal{L}\). 
Hence there exist \(K,K_1\in\mathbb{N}\) such that
\(
\omega\in\Omega_K\cap\mathcal{Z}_{K_1}.
\)
For every rational \(\varepsilon>0\), the preceding argument then gives
\(S_n^{\mathrm{int}}(\omega) \leq \varepsilon\)
for all sufficiently large \(n\).  
Hence
\[
        S_n^{\mathrm{int}}(\omega)
        \longrightarrow0.
\]
Thus
\(S_n^{\mathrm{int}}\to0\) almost surely on \(\{Z_\infty>0\}\).  
Equivalently,
\[
        \widehat{\mu}(\xi)
        \longrightarrow0
        \qquad
        (|\xi|\to\infty,\ \xi\in\mathbb{R})
\]
almost surely on non-extinction.
\end{proof}

\section{Endpoint-safe geometry for fixed nondegenerate arcs}\label{S:arc-geometry}

Throughout this section, let \(\gamma:[0,1]\to\mathbb{R}^2\) be a fixed nondegenerate \(C^2\) embedded arc. 
Thus \(\gamma\) is a \(C^2\) embedding and
\[
        \inf_{0\leq t\leq1}|\gamma'(t)|>0,
        \qquad
        \inf_{0\leq t\leq1}|\det(\gamma'(t),\gamma''(t))|>0.
\]
For \(\xi\in\mathbb{R}^2\setminus\{0\}\), set
\[
        \phi_\xi(t)=-2\pi\xi\cdot\gamma(t).
\]

Constants \(c_\gamma,C_\gamma\) may depend on \(\gamma\) and on the cutoff choices made below, 
but are independent of \(n,h,k,\ell,\xi,d\) and of the dyadic interval under consideration.
Constants \(C_{\gamma,j}\) may also depend on \(j\).

This section establishes the endpoint-safe fixed-cutoff package used in Section~\ref{S:arc-assembly}. 
Its deterministic part consists of the phase decomposition, 
the safe-support geometry, the surviving-band bounds, and the phase-bin coefficient estimates.
The remaining input is a weak-transfer estimate for the moving safe region. 
More precisely, we construct the partition, control the moving safe region
by a support estimate and weak transfer, bound the number of surviving bands and the shift of their phase scales, 
and prove the coefficient estimates and deterministic forcing bound required in the annular assembly. 
The endpoint pieces will always be placed inside the safe region.  
They will not be treated by integration by parts.

\subsection{Tangent angle and derivative sublevel geometry}\label{SS:arc-geometry-tangent-angle}

\begin{lemma}\label{L:arc-tangent-angle}
There exists a \(C^1\) function
\(\Theta_\gamma:[0,1]\to\mathbb{R}\) such that
\[
        \gamma'(t)
        =
        |\gamma'(t)|
        \bigl(\cos\Theta_\gamma(t),\sin\Theta_\gamma(t)\bigr),
        \qquad
        \Theta_\gamma'(t)
        =
        \frac{\det(\gamma'(t),\gamma''(t))}
        {|\gamma'(t)|^2}.
\]
There are constants \(0<c_\gamma\leq C_\gamma<\infty\) such that
\[
        c_\gamma\leq|\Theta_\gamma'(t)|\leq C_\gamma
        \qquad(0\leq t\leq1).
\]
Consequently, \(\Theta_\gamma\) is strictly monotone and bi-Lipschitz onto its image.
\end{lemma}

\begin{proof}
Since \(\gamma'\) is \(C^1\) and nonvanishing, 
the unit tangent \(T(t)=\frac{\gamma'(t)}{|\gamma'(t)|}\) is a \(C^1\) map from \([0,1]\) to \(\mathbb S^1\). 
Since \([0,1]\) is simply connected, \(T\) admits a \(C^1\) lift \(\Theta_\gamma:[0,1]\to\mathbb R\), so that
\[
        \gamma'(t)
        =
        |\gamma'(t)|(\cos\Theta_\gamma(t),\sin\Theta_\gamma(t)).
\]
Differentiating the tangent-angle representation and taking the determinant with \(\gamma'(t)\) gives
\[
        \det(\gamma'(t),\gamma''(t))
        =
        |\gamma'(t)|^2\Theta_\gamma'(t).
\]
Therefore
\[
        \Theta_\gamma'(t)
        =
        \frac{\det(\gamma'(t),\gamma''(t))}{|\gamma'(t)|^2}.
\]

Since \(\det(\gamma',\gamma'')\) is continuous and nonvanishing on \([0,1]\), 
it has constant sign and is bounded away from zero.  
Since \(|\gamma'|\) is bounded above and below by positive constants, the displayed formula yields
\[
        c_\gamma
        \leq
        |\Theta_\gamma'(t)|
        \leq
        C_\gamma
        \qquad(t\in[0,1]).
\]
Thus \(\Theta_\gamma\) is strictly monotone, and the same two-sided estimate gives the bi-Lipschitz property.
\end{proof}

For \(\xi\in\mathbb{R}^2\setminus\{0\}\), choose \(\alpha_\xi\in\mathbb{R}\) such that
\(\xi=|\xi|(\cos\alpha_\xi,\sin\alpha_\xi)\), and define
\[
        h_\xi(t)^{1/2}
        =
        \frac{|\xi\cdot\gamma'(t)|}
        {|\xi|\,|\gamma'(t)|}
        =
        \bigl|\cos(\Theta_\gamma(t)-\alpha_\xi)\bigr|.
\]

\begin{lemma}\label{L:arc-derivative-sublevel}
For every \(\xi\neq0\) and \(t\in[0,1]\),
\[
\begin{aligned}
        c_\gamma|\xi|h_\xi(t)^{1/2}
        &\leq
        |\phi_\xi'(t)|
        \leq
        C_\gamma|\xi|h_\xi(t)^{1/2},
\\
        |h_\xi'(t)|
        &\leq
        C_\gamma h_\xi(t)^{1/2}.
\end{aligned}
\]
Moreover, for every \(0<\rho\leq1\), the sublevel set
\[
\{t\in[0,1]:h_\xi(t)^{1/2}\leq\rho\}
\]
is contained in the union of at most \(C_\gamma\) intervals, each of length at most \(C_\gamma\rho\).
\end{lemma}

\begin{proof}
By definition,
\[
        |\phi_\xi'(t)|
        =
        2\pi|\xi\cdot\gamma'(t)|
        =
        2\pi|\xi|\,|\gamma'(t)|h_\xi(t)^{1/2}.
\]
The first estimate follows from the bounds on \(|\gamma'|\).
Differentiating the angular representation of \(h_\xi\) gives
\[
        h_\xi'(t)
        =
        -2\cos\bigl(\Theta_\gamma(t)-\alpha_\xi\bigr)
        \sin\bigl(\Theta_\gamma(t)-\alpha_\xi\bigr)
        \Theta_\gamma'(t).
\]
The bound on \(\Theta_\gamma'\) therefore yields
\[
        |h_\xi'(t)|
        \leq
        C_\gamma h_\xi(t)^{1/2}.
\]

For the sublevel estimate, set \(I_\xi=\Theta_\gamma([0,1])-\alpha_\xi\).
The interval \(I_\xi\) has length at most \(C_\gamma\).  
Since \(\operatorname{dist}(u,\pi/2+\pi\mathbb Z) \leq (\pi/2)|\cos u|\),
the set
\[
        \{u\in I_\xi:|\cos u|\leq\rho\}
\]
is contained in the union of at most \(C_\gamma\) intervals, each of length at most \(C\rho\).  
The bi-Lipschitz estimate for \(\Theta_\gamma\) gives the stated cover after pulling these intervals back to \([0,1]\).
\end{proof}

\subsection{Endpoint-safe partition}\label{SS:arc-geometry-partition}

We first define the decomposition for \(|\xi|\geq16\); 
the bounded-frequency case is specified below.  
Fix a nondecreasing \(\vartheta\in C^\infty(\mathbb R)\) such that
\(0\leq\vartheta\leq1\), \(\vartheta=0\) on \((-\infty,1]\), and \(\vartheta=1\) on \([2,\infty)\).  
Define
\[
        \eta_{\xi,0}(t)
        =
        \vartheta\bigl(|\xi|^{1/2}t\bigr),
        \qquad
        \eta_{\xi,1}(t)
        =
        \vartheta\bigl(|\xi|^{1/2}(1-t)\bigr).
\]
Set
\[
        \eta_\xi(t)
        =
        \eta_{\xi,0}(t)\eta_{\xi,1}(t).
\]
Then \(\|\eta_\xi'\|_{L^1([0,1])} \leq C\), where \(C\) depends only on \(\vartheta\).  
Define the endpoint-safe pieces by
\[
        \chi_{\xi,0}(t)
        =
        1-\eta_{\xi,0}(t),
        \qquad
        \chi_{\xi,1}(t)
        =
        \eta_{\xi,0}(t)
        \bigl(1-\eta_{\xi,1}(t)\bigr).
\]

Fix smooth functions
\(\psi,\zeta:[0,\infty)\to[0,1]\) such that
\[
        \psi\equiv1\ \text{on }[0,1],
        \qquad
        \operatorname{spt}\psi\subset[0,4],
        \qquad
        \operatorname{spt}\zeta\subset[1/4,4],
\]
and
\[
        \sum_{d\in2^{\mathbb{Z}}}
        \zeta\left(\frac{u}{d^2}\right)
        =
        1
        \qquad(u>0).
\]
Define the small-derivative and derivative-band cutoffs by
\[
\begin{aligned}
        \chi_{\xi,\mathrm{sd}}(t)
        &=
        \eta_\xi(t)\psi\bigl(|\xi|h_\xi(t)\bigr),
\\
        \chi_{\xi,d}(t)
        &=
        \eta_\xi(t)
        \bigl(1-\psi(|\xi|h_\xi(t))\bigr)
        \zeta\left(\frac{h_\xi(t)}{d^2}\right),
        \qquad d\in2^{\mathbb Z}.
\end{aligned}
\]
Let
\[
        \mathfrak D_{\xi}
        =
        \{d\in2^{\mathbb Z}:\chi_{\xi,d}\not\equiv0\}.
\]

For \(0<|\xi|<16\), set
\[
        \chi_{\xi,\mathrm{sd}}\equiv1,
        \qquad
        \chi_{\xi,0}\equiv0,
        \qquad
        \chi_{\xi,1}\equiv0,
        \qquad
        \mathfrak D_{\xi}=\emptyset.
\]

\begin{proposition}[Endpoint-safe partition]\label{P:arc-endpoint-safe-partition}
For every \(\xi\neq0\),
\[
        \chi_{\xi,0}(t)
        +
        \chi_{\xi,1}(t)
        +
        \chi_{\xi,\mathrm{sd}}(t)
        +
        \sum_{d\in \mathfrak D_{\xi}}\chi_{\xi,d}(t)
        =
        1
        \qquad(0\leq t\leq1).
\]
If \(|\xi|\geq16\), then
\[
\operatorname{spt}\chi_{\xi,0}\subset[0,2|\xi|^{-1/2}], \qquad       
\operatorname{spt}\chi_{\xi,1}\subset[1-2|\xi|^{-1/2},1],
\]
and
\[
        \operatorname{spt}\chi_{\xi,\mathrm{sd}}
        \subset
        \{t\in[0,1]:h_\xi(t)^{1/2}\leq2|\xi|^{-1/2}\}.
\]
For \(d\in \mathfrak D_{\xi}\),
\[
        \frac{d}{2}
        \leq
        h_\xi(t)^{1/2}
        \leq
        2d
        \qquad
        (t\in\operatorname{spt}\chi_{\xi,d}),
\]
and
\[
        c_\gamma|\xi|d
        \leq
        |\phi_\xi'(t)|
        \leq
        C_\gamma|\xi|d
        \qquad(t\in\operatorname{spt}\chi_{\xi,d}).
\]
Moreover,
\[
        \frac12|\xi|^{-1/2}
        \leq
        d
        \leq
        2
        \qquad(d\in \mathfrak D_{\xi}).
\]
In particular, \(\mathfrak D_{\xi}\) is finite.
\end{proposition}

\begin{proof}
The case \(0<|\xi|<16\) follows from the convention above, so assume \(|\xi|\geq16\). 
By definition, \(\chi_{\xi,0}+\chi_{\xi,1}=1-\eta_\xi\).
It remains to check the partition on the support of \(\eta_\xi\). 
If \(h_\xi(t)=0\), then \(\chi_{\xi,\mathrm{sd}}(t)=\eta_\xi(t)\) and
\(\chi_{\xi,d}(t)=0\) for every \(d\in2^{\mathbb Z}\). 
If \(h_\xi(t)>0\), then the dyadic partition of unity gives
\[
        \sum_{d\in2^{\mathbb{Z}}}
        \zeta\left(\frac{h_\xi(t)}{d^2}\right)=1.
\]
Thus
\[
        \chi_{\xi,\mathrm{sd}}(t)
        +
        \sum_{d\in \mathfrak D_{\xi}}\chi_{\xi,d}(t)
        =
        \eta_\xi(t).
\]
This proves the partition of unity.

The endpoint support bounds follow from the definitions of \(\eta_{\xi,0}\) and \(\eta_{\xi,1}\).  
Since \(\operatorname{spt}\psi\subset[0,4]\),
\[
        \operatorname{spt}\chi_{\xi,\mathrm{sd}}
        \subset
        \{t\in[0,1]:
        |\xi|h_\xi(t)\leq4\},
\]
which gives the stated small-derivative bound.

If \(t\in\operatorname{spt}\chi_{\xi,d}\), then
\[
        \frac{h_\xi(t)}{d^2}
        \in
        \operatorname{spt}\zeta
        \subset
        [1/4,4].
\]
Hence
\[
        \frac d2
        \leq
        h_\xi(t)^{1/2}
        \leq
        2d.
\]
The phase-derivative estimate now follows from Lemma~\ref{L:arc-derivative-sublevel}.

Finally, let \(d\in \mathfrak D_{\xi}\) and choose \(t\) with \(\chi_{\xi,d}(t)\neq0\).  
The band localization and \(h_\xi(t)\leq1\) give \(d\leq2\).  
Moreover, \(1-\psi(|\xi|h_\xi(t))\neq0\), and hence \(|\xi|h_\xi(t)>1\).  
Since \(h_\xi(t)^{1/2}\leq2d\),
\[
d \geq \frac12|\xi|^{-1/2}.
\]
Thus \(\mathfrak D_{\xi}\) is finite.
\end{proof}

\subsection{Fixed cutoff, safe support, and surviving oscillatory bands}\label{SS:arc-geometry-fixed-cutoff}

For \(\xi\in\mathbb{R}^2\setminus\{0\}\) and \(d\in \mathfrak D_\xi\), 
define the phase-bin scale \(m_{\xi,d}\in\mathbb{Z}\) by \(2^{m_{\xi,d}-1}<|\xi|d\leq2^{m_{\xi,d}}\).

Fix an integer \(j\geq1\).
Define
\[
        \mathfrak D_{\xi,j}^{\mathrm{small}}
        =
        \{d\in \mathfrak D_{\xi}:d<2^{-j}\},
        \qquad
        \mathfrak D_{\xi,j}^{\mathrm{osc}}
        =
        \{d\in \mathfrak D_{\xi}:d\geq2^{-j}\}.
\]
The fixed-cutoff safe function is
\[
        \Psi_{\xi,j}^{\mathrm{safe}}(t)
        =
        \chi_{\xi,0}(t)
        +
        \chi_{\xi,1}(t)
        +
        \chi_{\xi,\mathrm{sd}}(t)
        +
        \sum_{d\in \mathfrak D_{\xi,j}^{\mathrm{small}}}\chi_{\xi,d}(t).
\]

\begin{proposition}\label{P:arc-safe-support-and-bands}
Fix integers \(j,n\geq1\) and \(\xi\in\mathbb{R}^2\setminus\{0\}\) such that \(2^n\leq|\xi|<2^{n+1}\). 
Then \(\operatorname{spt}\Psi_{\xi,j}^{\mathrm{safe}}\) is contained in the union of at most \(C_\gamma\) intervals, 
each of length at most
\[
C_\gamma(2^{-j}+2^{-n/2}).
\]
Moreover,
\[
        \#\mathfrak D_{\xi,j}^{\mathrm{osc}} \leq j+2,
        \qquad
        |m_{\xi,d}-n| \leq \max\{j,2\} \quad (d\in \mathfrak D_{\xi,j}^{\mathrm{osc}}).
\]
\end{proposition}

\begin{proof}
If \(|\xi|<16\), then \(n\leq3\), \(\mathfrak D_{\xi}=\emptyset\), and \(\Psi_{\xi,j}^{\mathrm{safe}}\equiv1\).  
The conclusions follow after enlarging \(C_\gamma\), so we may assume \(|\xi|\geq16\). 

By Proposition~\ref{P:arc-endpoint-safe-partition},
\[
\begin{aligned}
        \operatorname{spt}\chi_{\xi,0}
        \cup
        \operatorname{spt}\chi_{\xi,1}
        \cup
        \operatorname{spt}\chi_{\xi,\mathrm{sd}}
        &\subset
        [0,2|\xi|^{-1/2}]
        \cup
        [1-2|\xi|^{-1/2},1]
\\
        &\quad\cup
        \{t\in[0,1]:
        h_\xi(t)^{1/2}\leq2|\xi|^{-1/2}\}.
\end{aligned}
\]
For \(d\in \mathfrak D_{\xi,j}^{\mathrm{small}}\),  Proposition~\ref{P:arc-endpoint-safe-partition} gives
\[
        \operatorname{spt}\chi_{\xi,d}
        \subset
        \{t\in[0,1]:h_\xi(t)^{1/2}\leq2d\}
        \subset
        \{t\in[0,1]:h_\xi(t)^{1/2}\leq2^{1-j}\}.
\]
Consequently,
\[
\begin{aligned}
        \operatorname{spt}\Psi_{\xi,j}^{\mathrm{safe}}
        &\subset
        [0,2|\xi|^{-1/2}]
        \cup
        [1-2|\xi|^{-1/2},1]
\\
        &\quad\cup
        \left\{
        t\in[0,1]:
        h_\xi(t)^{1/2}
        \leq
        2\bigl(2^{-j}+|\xi|^{-1/2}\bigr)
        \right\}.
\end{aligned}
\]
By Lemma~\ref{L:arc-derivative-sublevel}, the last set is contained in at most \(C_\gamma\) intervals, 
each of length at most
\[
        C_\gamma(2^{-j}+|\xi|^{-1/2}).
\]
If \(2(2^{-j}+|\xi|^{-1/2})\leq1\), Lemma~\ref{L:arc-derivative-sublevel} gives the required interval cover.
If this quantity exceeds \(1\), the trivial cover by \([0,1]\) gives the same conclusion after enlarging \(C_\gamma\).
Since \(|\xi|^{-1/2}\leq2^{-n/2}\), the asserted safe-support estimate follows.

For \(d\in \mathfrak D_{\xi,j}^{\mathrm{osc}}\), \(2^{-j}\leq d\leq2\).
There are therefore at most \(j+2\) possible dyadic values of \(d\).
Moreover,
\[
2^{n-j} \leq |\xi|d < 2^{n+2}.
\]
By the definition of \(m_{\xi,d}\), \(n-j \leq m_{\xi,d} \leq n+2\),
and hence
\[
|m_{\xi,d}-n| \leq \max\{j,2\}.
\]
\end{proof}

\subsection{Coefficient estimates and deterministic forcing}\label{SS:arc-geometry-coefficients}

For \(d\in \mathfrak D_{\xi}\), an integer \(\ell\geq0\), and \(J\in\mathcal{D}_{\ell+1}([0,1])\), define
\[
        c_J(\xi,d,\ell)
        =
        2^\ell
        \int_J
        \mathrm{e}^{i\phi_\xi(t)}
        \chi_{\xi,d}(t)
        \,\mathrm{d}t.
\]

\begin{lemma}\label{L:arc-variation-estimate}
For every \(\xi\neq0\) and \(d\in \mathfrak D_{\xi}\),
\[
        \int_{\operatorname{spt}\chi_{\xi,d}}
        \frac{|\chi_{\xi,d}'(t)|}{|\phi_\xi'(t)|}
        \,\mathrm{d}t
        \leq
        C_\gamma(|\xi|d)^{-1}.
\]
\end{lemma}

\begin{proof}
Since \(d\in \mathfrak D_{\xi}\), necessarily \(|\xi|\geq16\).  Set
\[
        U_{\xi,d}
        =
        \left\{
        t\in[0,1]:
        \frac d2
        \leq
        h_\xi(t)^{1/2}
        \leq
        2d
        \right\}.
\]
By the support properties of \(\zeta\) and \(\zeta'\),
\[
        \operatorname{spt}\chi_{\xi,d}'
        \subset 
        \operatorname{spt}\chi_{\xi,d}
        \subset
        U_{\xi,d}.
\]
On \(U_{\xi,d}\),
\[
        |\phi_\xi'(t)|\geq c_\gamma|\xi|d,
        \qquad
        |h_\xi'(t)|\leq C_\gamma d,
        \qquad
        |U_{\xi,d}|\leq C_\gamma d.
\]

Let
\[
        V_{\xi,d}
        =
        U_{\xi,d}
        \cap
        \{t\in[0,1]:1\leq|\xi|h_\xi(t)\leq4\}.
\]
Differentiating the definition of \(\chi_{\xi,d}\) and using the fixed cutoff bounds gives
\[
        |\chi_{\xi,d}'(t)|
        \leq
        |\eta_\xi'(t)|
        +
        C_\gamma d^{-2}|h_\xi'(t)|
        \mathbf 1_{U_{\xi,d}}(t)
        +
        C|\xi|\,|h_\xi'(t)|
        \mathbf 1_{V_{\xi,d}}(t).
\]
The first term has \(L^1\)-norm at most \(C\).  
On \(U_{\xi,d}\), \(|h_\xi'|\leq C_\gamma d\) and \(|U_{\xi,d}|\leq C_\gamma d\), so
\[
        d^{-2}
        \int_{U_{\xi,d}}|h_\xi'(t)|\,\mathrm{d}t
        \leq
        C_\gamma.
\]
On \(V_{\xi,d}\), one has
\(
        h_\xi^{1/2}\leq2|\xi|^{-1/2}.
\)
Lemma~\ref{L:arc-derivative-sublevel} therefore gives
\(
        |V_{\xi,d}|\leq C_\gamma|\xi|^{-1/2}
\)
and
\(
        |h_\xi'|\leq C_\gamma|\xi|^{-1/2}
\)
on \(V_{\xi,d}\).
It follows that
\[
        |\xi|
        \int_{V_{\xi,d}}|h_\xi'(t)|\,\mathrm{d}t
        \leq
        C_\gamma.
\]
Consequently,
\[
        \int_0^1|\chi_{\xi,d}'(t)|\,\mathrm{d}t
        \leq
        C_\gamma.
\]
Since
\[
        |\phi_\xi'(t)|
        \geq
        c_\gamma|\xi|d
        \qquad
        (t\in\operatorname{spt}\chi_{\xi,d}),
\]
the preceding variation bound gives
\[
        \int_{\operatorname{spt}\chi_{\xi,d}}
        \frac{|\chi_{\xi,d}'(t)|}{|\phi_\xi'(t)|}
        \,\mathrm{d}t
        \leq
        C_\gamma(|\xi|d)^{-1}.
\]
\end{proof}

\begin{proposition}\label{P:arc-coefficient-estimates}
There is \(C_\gamma<\infty\) such that, for every \(\xi\neq0\),
\(d\in \mathfrak D_{\xi}\), integer \(\ell\geq0\), 
and \(J\in\mathcal{D}_{\ell+1}([0,1])\),
\[
        |c_J(\xi,d,\ell)|
        \leq
        C_\gamma\min\{1,2^{\ell-m_{\xi,d}}\}.
\]
Moreover,
\[
        \left|
        \int_0^1
        \mathrm{e}^{i\phi_\xi(t)}
        \chi_{\xi,d}(t)
        \,\mathrm{d}t
        \right|
        \leq
        C_\gamma(|\xi|d)^{-1}.
\]
Consequently, for all integers \(j,n\geq1\) and
\(\xi\in\mathbb{R}^2\setminus\{0\}\) such that \(2^n\leq|\xi|<2^{n+1}\), one has 
\[
        \sum_{d\in \mathfrak D_{\xi,j}^{\mathrm{osc}}}
        \left|
        \int_0^1
        \mathrm{e}^{i\phi_\xi(t)}
        \chi_{\xi,d}(t)
        \,\mathrm{d}t
        \right|
        \leq
        C_{\gamma,j}2^{-n}.
\]
\end{proposition}

\begin{proof}
On \(\operatorname{spt}\chi_{\xi,d}\),
\[
        |\phi_\xi'(t)|
        \geq
        c_\gamma|\xi|d,
        \qquad
        |\phi_\xi''(t)|
        \leq
        C_\gamma|\xi|.
\]
Moreover,
\(
        |\operatorname{spt}\chi_{\xi,d}|
        \leq
        C_\gamma d.
\)

We first prove
\[
        \left|
        \int_J
        \mathrm{e}^{i\phi_\xi(t)}
        \chi_{\xi,d}(t)
        \,\mathrm{d}t
        \right|
        \leq
        C_\gamma(|\xi|d)^{-1}.
\]
Since \(\chi_{\xi,d}\) vanishes in a neighborhood of \(\{\phi_\xi'=0\}\), 
the quotient \( \frac{\chi_{\xi,d}}{\phi_\xi'}\) extends by zero to a \(C^1\) function on \([0,1]\).  
Hence
\[
\begin{aligned}
        \int_J
        \mathrm{e}^{i\phi_\xi(t)}
        \chi_{\xi,d}(t)\,\mathrm{d}t
        &=
        \left[
        \frac{\mathrm{e}^{i\phi_\xi(t)}
        \chi_{\xi,d}(t)}
        {i\phi_\xi'(t)}
        \right]_{\partial J}
\\
        &\quad-
        \int_J
        \mathrm{e}^{i\phi_\xi(t)}
        \left(
        \frac{\chi_{\xi,d}'(t)}{i\phi_\xi'(t)}
        -
        \frac{\chi_{\xi,d}(t)\phi_\xi''(t)}
        {i(\phi_\xi'(t))^2}
        \right)
        \,\mathrm{d}t.
\end{aligned}
\]
The boundary term is bounded by \(C_\gamma(|\xi|d)^{-1}\).

The \(\chi_{\xi,d}'\)-term is bounded by Lemma~\ref{L:arc-variation-estimate}.  
The \(\phi_\xi''\)-term is bounded by
\[
C_\gamma \int_{\operatorname{spt}\chi_{\xi,d}} \frac{|\xi|}{|\xi|^2d^2} \,\mathrm{d}t \leq C_\gamma \frac{|\xi|}{|\xi|^2d^2}\cdot d = C_\gamma(|\xi|d)^{-1}.
\]
This proves the integral estimate on \(J\). 
Multiplying by \(2^\ell\), using \(2^{m_{\xi,d}-1}<|\xi|d\), 
and combining with the trivial estimate \(2^\ell|J|=1/2\), we obtain
\[
        |c_J(\xi,d,\ell)|
        \leq
        C_\gamma
        \min\{1,2^{\ell-m_{\xi,d}}\}.
\]
The same argument with \([0,1]\) in place of \(J\) gives
\[
        \left|
        \int_0^1
        \mathrm{e}^{i\phi_\xi(t)}
        \chi_{\xi,d}(t)
        \,\mathrm{d}t
        \right|
        \leq
        C_\gamma(|\xi|d)^{-1}.
\]
If \(d\in \mathfrak D_{\xi,j}^{\mathrm{osc}}\), then \(d\geq2^{-j}\), and
Proposition~\ref{P:arc-safe-support-and-bands} gives \(\#\mathfrak D_{\xi,j}^{\mathrm{osc}}\leq j+2\).
Therefore, for \(2^n\leq|\xi|<2^{n+1}\),
\[
\begin{aligned}
        \sum_{d\in \mathfrak D_{\xi,j}^{\mathrm{osc}}}
        \left|
        \int_0^1
        \mathrm{e}^{i\phi_\xi(t)}
        \chi_{\xi,d}(t)
        \,\mathrm{d}t
        \right|
        &\leq
        C_\gamma
        \sum_{d\in \mathfrak D_{\xi,j}^{\mathrm{osc}}}
        (|\xi|d)^{-1}
\\
        &\leq
        C_{\gamma,j}|\xi|^{-1}
\\
        &\leq
        C_{\gamma,j}2^{-n}.
\end{aligned}
\]
The proof is complete.
\end{proof}

\subsection{Small-interval transfer for the moving safe region}\label{SS:arc-geometry-safe-transfer}

The safe support varies with \(\xi\) and \(n\), 
so atomlessness of \(\mu\) must be transferred to the approximants \(\mu_{n+h}\).  
The next lemma gives the required uniform form.

\begin{lemma}\label{L:arc-small-interval-transfer}
For all integers \(h\geq0\) and \(j\geq1\),
\[
\limsup_{n\to\infty}\sup_{2^n\leq|\xi|<2^{n+1}} \mu_{n+h}\bigl(\operatorname{spt}\Psi_{\xi,j}^{\mathrm{safe}}\bigr) \leq C_\gamma\Delta_{C_\gamma2^{-j}}(\mu)
\]
almost surely on non-extinction, where
\[
        \Delta_\rho(\mu)
        =
        \sup\{\mu(I):I\subset[0,1]\text{ is an interval and }|I|\leq\rho\}.
\]
Consequently, for every integer \(h\geq0\),
\[
        \lim_{j\to\infty}
        \limsup_{n\to\infty}
        \sup_{2^n\leq|\xi|<2^{n+1}}
        \mu_{n+h}
        \bigl(\operatorname{spt}\Psi_{\xi,j}^{\mathrm{safe}}\bigr)
        =
        0
\]
almost surely on non-extinction. 
\end{lemma}

\begin{proof}
By Proposition~\ref{P:arc-safe-support-and-bands}, for \(2^n\leq|\xi|<2^{n+1}\), 
the set \(\operatorname{spt}\Psi_{\xi,j}^{\mathrm{safe}}\) is contained in the union of at most \(C_\gamma\) intervals of length at most \(C_\gamma(2^{-j}+2^{-n/2})\).
Therefore
\[
\sup_{2^n\leq|\xi|<2^{n+1}} \mu_{n+h} \bigl(\operatorname{spt}\Psi_{\xi,j}^{\mathrm{safe}}\bigr) \leq C_\gamma \Delta_{C_\gamma(2^{-j}+2^{-n/2})}(\mu_{n+h}).
\]
Fix \(j\).  For all sufficiently large \(n\), \(2^{-n/2}\leq2^{-j}\).  
Since \(\mu_{n+h}\Rightarrow\mu\), the monotonicity of the small-interval modulus 
and Lemma~\ref{L:preliminaries-small-interval-modulus} give
\[
\begin{aligned}
        \limsup_{n\to\infty}
        &\sup_{2^n\leq|\xi|<2^{n+1}}
        \mu_{n+h}
        \bigl(\operatorname{spt}\Psi_{\xi,j}^{\mathrm{safe}}\bigr)
\\
        &\leq
        C_\gamma
        \limsup_{n\to\infty}
        \Delta_{2C_\gamma2^{-j}}(\mu_{n+h})
\\
        &\leq
        C_\gamma
        \Delta_{4C_\gamma2^{-j}}(\mu).
\end{aligned}
\]
After enlarging \(C_\gamma\), this is the asserted bound.
On non-extinction, \(\mu\) is atomless, so \(\Delta_{C_\gamma2^{-j}}(\mu)\to0\) as \(j\to\infty\).  
This gives the second assertion.
\end{proof}

\subsection{Endpoint-safe fixed-cutoff package for the annular assembly}
\label{SS:arc-geometry-output}

The preceding results are summarized in the following proposition.

\begin{proposition}[Fixed-cutoff package]
\label{P:arc-fixed-cutoff-package}

For every integer \(j\geq1\) and every
\(\xi\in\mathbb{R}^2\setminus\{0\}\),
\[
        \Psi_{\xi,j}^{\mathrm{safe}}
        +
        \sum_{d\in \mathfrak D_{\xi,j}^{\mathrm{osc}}}\chi_{\xi,d}
        =
        1
        \qquad\text{on }[0,1].
\]

For every fixed integer \(h\geq0\), almost surely on non-extinction, and for every \(j\geq1\),
\[
        \limsup_{n\to\infty}
        \sup_{2^n\leq|\xi|<2^{n+1}}
        \mu_{n+h}
        \bigl(\operatorname{spt}\Psi_{\xi,j}^{\mathrm{safe}}\bigr)
        \leq
        C_\gamma\Delta_{C_\gamma2^{-j}}(\mu).
\]
Moreover, this upper bound tends to zero as \(j\to\infty\) on non-extinction.

Suppose that \(n\geq1\) and \(2^n\leq|\xi|<2^{n+1}\).  
Then
\[
        \#\mathfrak D_{\xi,j}^{\mathrm{osc}}
        \leq
        C_{\gamma,j},
        \qquad
        |m_{\xi,d}-n|
        \leq
        C_{\gamma,j}
        \quad
        \bigl(d\in \mathfrak D_{\xi,j}^{\mathrm{osc}}\bigr).
\]
For every \(d\in \mathfrak D_{\xi,j}^{\mathrm{osc}}\), every \(\ell\geq0\),
and every \(J\in\mathcal{D}_{\ell+1}([0,1])\),
\[
        |c_J(\xi,d,\ell)|
        \leq
        C_\gamma\min\{1,2^{\ell-m_{\xi,d}}\}.
\]
Moreover,
\[
        \sum_{d\in \mathfrak D_{\xi,j}^{\mathrm{osc}}}
        \left|
        \int_0^1
        \mathrm{e}^{i\phi_\xi(t)}
        \chi_{\xi,d}(t)
        \,\mathrm{d}t
        \right|
        \leq
        C_{\gamma,j}2^{-n}.
\]

Finally, let \(d\in \mathfrak D_{\xi,j}^{\mathrm{osc}}\), \(k\geq1\), 
and \(J\in\mathcal{D}_k([0,1])\), and set
\[
        \Gamma_J(\xi,d)
        =
        A_{J^-}c_J(\xi,d,k-1).
\]
Then \(\Gamma_J(\xi,d)\) is \(\mathcal{F}_{k-1}\)-measurable and satisfies
\[
\begin{aligned}
        |\Gamma_J(\xi,d)|
        &\leq
        C_\gamma A_{J^-}
        \min\{1,2^{k-m_{\xi,d}}\},
\\
        |\Gamma_J(\xi,d)|W_J
        &\leq
        C_\gamma A_J
        \min\{1,2^{k-m_{\xi,d}}\}.
\end{aligned}
\]
\end{proposition}

\begin{proof}
The partition identity follows from Proposition~\ref{P:arc-endpoint-safe-partition} 
and the definitions of \(\mathfrak D_{\xi,j}^{\mathrm{small}}\), \(\mathfrak D_{\xi,j}^{\mathrm{osc}}\), 
and \(\Psi_{\xi,j}^{\mathrm{safe}}\).
The safe-region estimate is Lemma~\ref{L:arc-small-interval-transfer}.
The band number and phase-shift bounds follow from Proposition~\ref{P:arc-safe-support-and-bands}, 
while the coefficient and deterministic-forcing estimates follow from Proposition~\ref{P:arc-coefficient-estimates}.

For the final assertion, \(c_J(\xi,d,k-1)\) is deterministic and \(A_{J^-}\) is \(\mathcal{F}_{k-1}\)-measurable.  
Hence \(\Gamma_J(\xi,d)\) is \(\mathcal{F}_{k-1}\)-measurable.
Applying the coefficient estimate with \(\ell=k-1\) gives the first bound.  
The second follows from \(A_{J^-}W_J=2A_J\), after absorbing the factor \(2\) into \(C_\gamma\).
\end{proof}

\begin{remark}\label{R:arc-fixed-cutoff-necessity}
The fixed cutoff keeps both the number of surviving bands and the phase-scale shifts uniformly bounded on each annulus.  
Without it, the number of active bands may grow with \(n\), while \(m_{\xi,d}\) may drift far below \(n\).  
The qualitative lower-deviation estimates of Section~\ref{S:spine-engine} provide no rate that absorbs either loss.
\end{remark}

\begin{remark}\label{R:arc-endpoint-safe-principle}
The endpoint pieces \(\chi_{\xi,0}\) and \(\chi_{\xi,1}\) remain in \(\Psi_{\xi,j}^{\mathrm{safe}}\).  
They are controlled by atomlessness and weak transfer rather than by integration by parts.  
This is the endpoint-specific feature of the fixed-arc argument.
\end{remark}

\section{Annular assembly for fixed nondegenerate arcs}\label{S:arc-assembly}

Let \(\gamma:[0,1]\to\mathbb{R}^2\) be a fixed nondegenerate \(C^2\) embedded arc, 
and let \(\mu_\gamma=\gamma_\#\mu\).
For \(\xi\in\mathbb{R}^2\setminus\{0\}\), recall that \(\phi_\xi(t)=-2\pi\xi\cdot\gamma(t)\).
For \(n\geq1\), define the annular supremum
\[
S_{\gamma,n}
=
\sup_{2^n\leq|\xi|<2^{n+1}}
|\widehat{\mu_\gamma}(\xi)|.
\]
We prove that
\[
S_{\gamma,n}\to0 \qquad(n\to\infty)
\]
almost surely on \(\{\mu([0,1])>0\}\).

We work throughout on a common probability-one event on which:
\[
        \mu_n\Rightarrow\mu,\qquad
        Z_n\to Z_\infty<\infty,
\]
all descendant limits \(Z_\infty^{(u)}\) exist,
\[
        \mu(I_u)=A_uZ_\infty^{(u)}
        \qquad(u\in\{0,1\}^{\ast}),
\]
dyadic endpoints have zero \(\mu\)-mass, 
the terminal descendant theorem over planar grids holds for all countable parameter choices, 
the abstract stopped skeleton theorem holds for all countable parameter choices, 
and the endpoint-safe fixed-cutoff package in Proposition~\ref{P:arc-fixed-cutoff-package}  is available for every integer cutoff \(j\geq1\).

For \(K,K_1\in\mathbb{N}\), recall the localization events
\[
\Omega_K=\{Z_\infty\leq K\},
\qquad
\text{and}
\qquad
\mathcal{Z}_{K_1}
=
\left\{\sup_{\ell\geq0}Z_\ell\leq K_1\right\}.
\]
Since \(Z_n\to Z_\infty<\infty\) almost surely,
\[
        \mathbb{P}
        \left(
        \bigcup_{K=1}^{\infty}
        \bigcup_{K_1=1}^{\infty}
        \Omega_K\cap\mathcal{Z}_{K_1}
        \right)=1.
\]

\subsection{Original-transform gridding}\label{SS:arc-assembly-original-grid}

Set \(R_\gamma := \max\{1,\sup_{t\in[0,1]}|\gamma(t)|\}\).

\begin{lemma}\label{L:arc-original-grid}
Fix \(K\geq1\) and \(\varepsilon>0\).  
For every \(n\geq1\), there is a deterministic finite set
\[
        \Lambda_n(K,\varepsilon)
        \subset
        \{\xi\in\mathbb{R}^2:2^n\leq|\xi|<2^{n+1}\}
\]
such that
\[
        \#\Lambda_n(K,\varepsilon)
        \leq
        C_\gamma
        \bigl(1+K^2\varepsilon^{-2}\bigr)2^{2n}, 
\]
and, on \(\Omega_K\),
\[
        S_{\gamma,n}>\varepsilon
        \quad\Longrightarrow\quad
        \max_{\xi\in\Lambda_n(K,\varepsilon)}
        |\widehat{\mu_\gamma}(\xi)|
        >
        \frac{31}{32}\varepsilon.
\]
\end{lemma}

\begin{proof}
For \(\xi,\eta\in\mathbb{R}^2\), on \(\Omega_K\),
\[
\begin{aligned}
        |\widehat{\mu_\gamma}(\xi)-\widehat{\mu_\gamma}(\eta)|
        &\leq
        \int_0^1
        \left|
        \mathrm{e}^{-2\pi i\xi\cdot\gamma(t)}
        -
        \mathrm{e}^{-2\pi i\eta\cdot\gamma(t)}
        \right|
        \,\mathrm{d}\mu(t)
\\
        &\leq
        2\pi R_\gamma|\xi-\eta|\mu([0,1])
\\
        &\leq
        2\pi R_\gamma K|\xi-\eta|.
\end{aligned}
\]
Choose \(\Lambda_n(K,\varepsilon)\) to be a maximal 
\(\frac{\varepsilon}{64\pi R_\gamma K}\)-separated subset of \(\{\xi\in\mathbb{R}^2:2^n\leq|\xi|<2^{n+1}\}\).
Then \(\Lambda_n(K,\varepsilon)\) is a \(\frac{\varepsilon}{64\pi R_\gamma K}\)-net of the annulus.
A standard packing estimate gives
\[
        \#\Lambda_n(K,\varepsilon)
        \leq
        C_\gamma
        \bigl(1+K^2\varepsilon^{-2}\bigr)2^{2n}.
\]
If \(S_{\gamma,n}>\varepsilon\), choose \(\eta\) in the annulus with
\(
        |\widehat{\mu_\gamma}(\eta)|>\varepsilon.
\)
Choose \(\xi\in\Lambda_n(K,\varepsilon)\) with \(|\xi-\eta|\leq\frac{\varepsilon}{64\pi R_\gamma K}\).
The Lipschitz estimate gives
\[
        |\widehat{\mu_\gamma}(\xi)|
        \geq
        |\widehat{\mu_\gamma}(\eta)|
        -
        |\widehat{\mu_\gamma}(\xi)-\widehat{\mu_\gamma}(\eta)|
        >
        \varepsilon-\frac{\varepsilon}{32}
        =
        \frac{31}{32}\varepsilon.
\]
This proves the lemma.
\end{proof}

\begin{remark}
The gridding in Lemma~\ref{L:arc-original-grid} is applied only to the original transform
\(
\widehat{\mu_\gamma}(\xi).
\)
The moving cutoffs \( \Psi_{\xi,j}^{\mathrm{safe}}, \chi_{\xi,d}\) depend on \(\xi\), 
and we never use Lipschitz estimates in \(\xi\) for the decomposed pieces.
\end{remark}

\subsection{Atomization and terminal descendants}\label{SS:arc-assembly-atomization}

Choose once and for all a deterministic point \(t_u\in I_u\) for every finite word \(u\).
Since \(\gamma\) is Lipschitz, for every \(n\geq1\), \(h\in\mathbb{N}_0\), \(|u|=n+h\), \(t\in I_u\), 
and \(2^n\leq|\xi|<2^{n+1}\),
\[
        |\phi_\xi(t)-\phi_\xi(t_u)|
        \leq
        C_\gamma2^{-h}.
\]

\begin{lemma}\label{L:arc-limiting-atomization}
Fix \(K\geq1\) and \(h\in\mathbb{N}_0\).  Then, on \(\Omega_K\), for every \(n\geq1\),
\[
        \left|
        \widehat{\mu_\gamma}(\xi)
        -
        \sum_{|u|=n+h}
        \mu(I_u)\mathrm{e}^{i\phi_\xi(t_u)}
        \right|
        \leq
        C_\gamma K2^{-h}
\]
uniformly over \(2^n\leq|\xi|<2^{n+1}\).
\end{lemma}

\begin{proof}
By the preceding phase estimate, on \(\Omega_K\),
\[
\begin{aligned}
        \left|
        \widehat{\mu_\gamma}(\xi)
        -
        \sum_{|u|=n+h}
        \mu(I_u)\mathrm{e}^{i\phi_\xi(t_u)}
        \right|
        &\leq
        \sum_{|u|=n+h}
        \int_{I_u}
        \left|
        \mathrm{e}^{i\phi_\xi(t)}
        -
        \mathrm{e}^{i\phi_\xi(t_u)}
        \right|
        \,\mathrm{d}\mu(t)
\\
        &\leq
        C_\gamma2^{-h}\mu([0,1])
        \leq
        C_\gamma K2^{-h}.
\end{aligned}
\]
\end{proof}

By the cylinder identity \(\mu(I_u)=A_uZ_\infty^{(u)}\), we have
\[
\sum_{|u|=n+h}
\mu(I_u)\mathrm{e}^{i\phi_\xi(t_u)}
=
\sum_{|u|=n+h}
A_u\mathrm{e}^{i\phi_\xi(t_u)}
+
\sum_{|u|=n+h}
A_u(Z_\infty^{(u)}-1)\mathrm{e}^{i\phi_\xi(t_u)}.
\]

\begin{lemma}[Terminal descendant control on the grid]
\label{L:arc-terminal-descendant-control}
Fix \(K,K_1\in\mathbb{N}\), \(\varepsilon\in\mathbb{Q}_{>0}\), and \(h\in\mathbb{N}_0\).
Then
\[
\sup_{\xi\in\Lambda_n(K,\varepsilon)}
\left|
\sum_{|u|=n+h}
A_u
\bigl(Z_\infty^{(u)}-1\bigr)
\mathrm{e}^{i\phi_\xi(t_u)}
\right|
\longrightarrow0
\qquad
(n\to\infty)
\]
almost surely on \(\mathcal{Z}_{K_1} \cap \{Z_\infty>0\}\).
\end{lemma}

\begin{proof}
By Lemma~\ref{L:arc-original-grid}, the grids \(\Lambda_n(K,\varepsilon)\) satisfy the cardinality hypothesis of
Theorem~\ref{T:terminal-descendants-planar}.
Apply Theorem~\ref{T:terminal-descendants-planar} with \(\gamma_{u,n}(\xi)=\mathrm{e}^{i\phi_\xi(t_u)}\).
These coefficients are deterministic and have modulus one.
\end{proof}

\begin{lemma}\label{L:arc-level-atomization}
Fix \(K_1\geq1\) and \(h\in\mathbb{N}_0\).  
Then, on \(\mathcal{Z}_{K_1}\), for every \(n\geq1\),
\[
        \left|
        \sum_{|u|=n+h}
        A_u\mathrm{e}^{i\phi_\xi(t_u)}
        -
        \int_0^1
        \mathrm{e}^{i\phi_\xi(t)}
        \,\mathrm{d}\mu_{n+h}(t)
        \right|
        \leq
        C_\gamma K_1 2^{-h}
\]
uniformly over \(2^n\leq|\xi|<2^{n+1}\).
\end{lemma}

\begin{proof}
Since \(\mu_{n+h}(I_u)=A_u\), the preceding phase estimate gives, on \(\mathcal{Z}_{K_1}\),
\[
\begin{aligned}
        \left|
        \sum_{|u|=n+h}
        A_u\mathrm{e}^{i\phi_\xi(t_u)}
        -
        \int_0^1
        \mathrm{e}^{i\phi_\xi(t)}
        \,\mathrm{d}\mu_{n+h}(t)
        \right|
        &\leq
        C_\gamma2^{-h}\mu_{n+h}([0,1])
\\
        &=
        C_\gamma2^{-h}Z_{n+h}
        \leq
        C_\gamma K_1 2^{-h}.
\end{aligned}
\]
\end{proof}

\subsection{Fixed-cutoff decomposition of the level integral}\label{SS:arc-assembly-fixed-cutoff}
For integers \(h\geq0\), \(j,n\geq1\), and \(\xi\in\mathbb{R}^2\setminus\{0\}\) satisfying \(2^n\leq|\xi|<2^{n+1}\), 
Proposition~\ref{P:arc-fixed-cutoff-package} gives
\[
1=
\Psi_{\xi,j}^{\mathrm{safe}}
+
\sum_{d\in \mathfrak D_{\xi,j}^{\mathrm{osc}}}\chi_{\xi,d}.
\]
Thus
\[
        \int_0^1
        \mathrm{e}^{i\phi_\xi(t)}
        \,\mathrm{d}\mu_{n+h}(t)
        =
        E_{\xi,n,h,j}^{\mathrm{safe}}
        +
        \sum_{d\in \mathfrak D_{\xi,j}^{\mathrm{osc}}}
        \int_0^1
        \mathrm{e}^{i\phi_\xi(t)}
        \chi_{\xi,d}(t)
        \,\mathrm{d}\mu_{n+h}(t),
\]
where
\[
        E_{\xi,n,h,j}^{\mathrm{safe}}
        \coloneqq
        \int_0^1
        \mathrm{e}^{i\phi_\xi(t)}
        \Psi_{\xi,j}^{\mathrm{safe}}(t)
        \,\mathrm{d}\mu_{n+h}(t).
\]
The safe term satisfies
\[
        |E_{\xi,n,h,j}^{\mathrm{safe}}|
        \leq
        \mu_{n+h}
        \bigl(\operatorname{spt}\Psi_{\xi,j}^{\mathrm{safe}}\bigr).
\]

\begin{lemma}
\label{L:arc-safe-region-control}
Fix an integer \(h\geq0\).  
Almost surely on non-extinction, for every integer \(j\geq1\),
\[
        \limsup_{n\to\infty}
        \sup_{2^n\leq|\xi|<2^{n+1}}
        |E_{\xi,n,h,j}^{\mathrm{safe}}|
        \leq
        C_\gamma\Delta_{C_\gamma2^{-j}}(\mu).
\]
Consequently,
\[
        \lim_{j\to\infty}
        \limsup_{n\to\infty}
        \sup_{2^n\leq|\xi|<2^{n+1}}
        |E_{\xi,n,h,j}^{\mathrm{safe}}|
        =
        0
\]
almost surely on non-extinction.
\end{lemma}

\begin{proof}
The first assertion follows from the preceding estimate and
Lemma~\ref{L:arc-small-interval-transfer}.
The second follows from the atomlessness of \(\mu\) on non-extinction
and Lemma~\ref{L:preliminaries-small-interval-modulus}.
\end{proof}

\subsection{Martingale skeleton}
\label{SS:arc-assembly-skeleton}

For \(d\in \mathfrak D_{\xi,j}^{\mathrm{osc}}\), \(\ell\geq0\), and 
\(J\in\mathcal{D}_{\ell+1}([0,1])\), recall that
\[
        c_J(\xi,d,\ell)
        =
        2^\ell
        \int_J
        \mathrm{e}^{i\phi_\xi(t)}
        \chi_{\xi,d}(t)
        \,\mathrm{d}t.
\]

\begin{lemma}[Finite-level martingale identity]\label{L:arc-finite-level-martingale-identity}
For integers \(n,j\geq1\) and \(h\geq0\), every \(\xi\in\mathbb{R}^2\setminus\{0\}\), 
and every \(d\in \mathfrak D_{\xi,j}^{\mathrm{osc}}\), one has
\[
        \int_0^1
        \mathrm{e}^{i\phi_\xi(t)}\chi_{\xi,d}(t)
        \,\mathrm{d}\mu_{n+h}(t)
        =
        \int_0^1
        \mathrm{e}^{i\phi_\xi(t)}\chi_{\xi,d}(t)
        \,\mathrm{d}t
        +
        \sum_{\ell=0}^{n+h-1}
        \sum_{J\in\mathcal{D}_{\ell+1}([0,1])}
        A_{J^-}(W_J-1)c_J(\xi,d,\ell).
\]
\end{lemma}

\begin{proof}
By telescoping and since \(\mu_0\) is Lebesgue measure on \([0,1]\),
\[
\begin{aligned}
        \int_0^1
        \mathrm{e}^{i\phi_\xi(t)}
        \chi_{\xi,d}(t)
        \,\mathrm{d}\mu_{n+h}(t)
        &=
        \int_0^1
        \mathrm{e}^{i\phi_\xi(t)}
        \chi_{\xi,d}(t)
        \,\mathrm{d}t
\\
        &\quad+
        \sum_{\ell=0}^{n+h-1}
        \int_0^1
        \mathrm{e}^{i\phi_\xi(t)}
        \chi_{\xi,d}(t)
        \,\mathrm{d}(\mu_{\ell+1}-\mu_\ell)(t).
\end{aligned}
\]
Fix \(0\leq\ell\leq n+h-1\) and \(J\in\mathcal{D}_{\ell+1}([0,1])\), 
and let \(J^-\in\mathcal{D}_{\ell}([0,1])\) be its parent.  
The level-\(\ell\) density on \(J\) is \(2^\ell A_{J^-}\).
Since \(A_J=\frac12A_{J^-}W_J\), the level-\((\ell+1)\) density on \(J\) is \(2^\ell A_{J^-}W_J\).
Hence, as an identity of densities,
\[
        \mathrm{d}(\mu_{\ell+1}-\mu_\ell)(t)
        =
        \sum_{J\in\mathcal{D}_{\ell+1}([0,1])}
        2^\ell A_{J^-}(W_J-1)
        \mathbf{1}_J(t)\,\mathrm{d}t.
\]
Therefore
\[
\begin{aligned}
        \int_J\mathrm{e}^{i\phi_\xi(t)}\chi_{\xi,d}(t)\,\mathrm{d}(\mu_{\ell+1}-\mu_\ell)(t)
        & =
        A_{J^-}(W_J-1)
        2^\ell
        \int_J\mathrm{e}^{i\phi_\xi(t)}\chi_{\xi,d}(t)\,\mathrm{d}t
\\
        & =
        A_{J^-}(W_J-1)c_J(\xi,d,\ell).
\end{aligned}
\]
Summing over all \(J\in\mathcal{D}_{\ell+1}([0,1])\) and then over \(0\leq\ell\leq n+h-1\) proves the identity.
\end{proof}

We now sum the finite-level martingale identity over
\(
        d\in \mathfrak D_{\xi,j}^{\mathrm{osc}}.
\)
For \(n,j\geq1\) and \(2^n\leq|\xi|<2^{n+1}\), define the deterministic forcing by
\[
        D_{\xi,n,j}^{(0)}
        =
        \sum_{d\in \mathfrak D_{\xi,j}^{\mathrm{osc}}}
        \int_0^1
        \mathrm{e}^{i\phi_\xi(t)}
        \chi_{\xi,d}(t)
        \,\mathrm{d}t.
\]
Proposition~\ref{P:arc-fixed-cutoff-package} gives
\begin{equation}
\label{E:arc-deterministic-forcing-bound}
        \sup_{2^n\leq|\xi|<2^{n+1}}
        \left|
        D_{\xi,n,j}^{(0)}
        \right|
        \leq
        C_{\gamma,j}2^{-n}.
\end{equation}

Define the oscillatory skeleton
\[
        S_{\xi,j,n,h}
        =
        \sum_{d\in \mathfrak D_{\xi,j}^{\mathrm{osc}}}
        \sum_{\ell=0}^{n+h-1}
        \sum_{J\in\mathcal{D}_{\ell+1}([0,1])}
        A_{J^-}(W_J-1)c_J(\xi,d,\ell).
\]
The finite-level identity therefore yields
\[
\sum_{d\in \mathfrak D_{\xi,j}^{\mathrm{osc}}} \int_0^1 \mathrm{e}^{i\phi_\xi(t)} \chi_{\xi,d}(t) \,\mathrm{d}\mu_{n+h}(t) = D_{\xi,n,j}^{(0)} + S_{\xi,j,n,h}.
\]

\begin{lemma}[Oscillatory skeleton control on the grid]
\label{L:arc-oscillatory-skeleton-control}
Fix \(K,K_1\in\mathbb{N}\), \(\varepsilon\in\mathbb{Q}_{>0}\), and integers \(j\geq1\), \(h\geq0\). 
Then
\[
        \sup_{\xi\in\Lambda_n(K,\varepsilon)}
        |S_{\xi,j,n,h}|
        \longrightarrow0
\]
almost surely on \(\mathcal{Z}_{K_1}\).
\end{lemma}

\begin{proof}
For \(1\leq k\leq n+h\), \(J\in\mathcal{D}_k([0,1])\), and \(d\in \mathfrak D_{\xi,j}^{\mathrm{osc}}\), 
define \(\Gamma_J(\xi,d) = A_{J^-}c_J(\xi,d,k-1)\).
Then \(\Gamma_J(\xi,d)\) is \(\mathcal F_{k-1}\)-measurable, and
Proposition~\ref{P:arc-coefficient-estimates} gives
\[
        |\Gamma_J(\xi,d)|
        \leq
        C_\gamma A_{J^-}
        \min\{1,2^{k-m_{\xi,d}}\}.
\]
Using \(A_J=\frac{1}{2}A_{J^-}W_J\),
we also obtain
\[
        |\Gamma_J(\xi,d)|W_J
        \leq
        C_\gamma
        A_J
        \min\{1,2^{k-m_{\xi,d}}\},
\]
after enlarging \(C_\gamma\) if necessary.

For fixed \(j\), Proposition~\ref{P:arc-safe-support-and-bands} gives
\[
        \#\mathfrak D_{\xi,j}^{\mathrm{osc}}
        \leq
        C_{\gamma,j},
        \qquad
        |m_{\xi,d}-n|
        \leq
        C_{\gamma,j}.
\]
Lemma~\ref{L:arc-original-grid} gives the required planar-grid cardinality bound.

Thus Theorem~\ref{T:abstract-stopped-skeleton} applies with \(\Lambda_n=\Lambda_n(K,\varepsilon)\) 
and \(\mathcal{B}_{\xi,n}=\mathfrak D_{\xi,j}^{\mathrm{osc}}\), 
with \(M_0\) and \(C_0\) depending only on \(\gamma\) and \(j\).
The asserted convergence follows.
\end{proof}

\subsection{Full annular decomposition}
\label{SS:arc-assembly-full-decomposition}

\begin{lemma}
\label{L:arc-full-annular-decomposition}
Fix \(K,K_1\in\mathbb{N}\), \(\varepsilon\in\mathbb{Q}_{>0}\), \(j\in\mathbb{N}\), and \(h\in\mathbb{N}_0\).
Then, for every \(n\geq1\) and \(\xi\in\Lambda_n(K,\varepsilon)\), one has
\[
\begin{aligned}
\widehat{\mu_\gamma}(\xi)
&=
\left[
\widehat{\mu_\gamma}(\xi)
-
\sum_{|u|=n+h}
\mu(I_u)\mathrm{e}^{i\phi_\xi(t_u)}
\right]
+
\sum_{|u|=n+h}
A_u
\bigl(Z_\infty^{(u)}-1\bigr)
\mathrm{e}^{i\phi_\xi(t_u)}
\\
&\quad+
\left[
\sum_{|u|=n+h}
A_u\mathrm{e}^{i\phi_\xi(t_u)}
-
\int_0^1
\mathrm{e}^{i\phi_\xi(t)}
\,\mathrm{d}\mu_{n+h}(t)
\right]
+
E_{\xi,n,h,j}^{\mathrm{safe}}
+
D_{\xi,n,j}^{(0)}
+
S_{\xi,j,n,h}.
\end{aligned}
\]

Moreover, almost surely on
\(\Omega_K \cap \mathcal{Z}_{K_1} \cap \{Z_\infty>0\}\), one has
\[
\limsup_{n\to\infty}
\sup_{\xi\in\Lambda_n(K,\varepsilon)}
\left|
\widehat{\mu_\gamma}(\xi)
\right|
\leq
C_\gamma(K+K_1)2^{-h}
+
C_\gamma
\Delta_{C_\gamma2^{-j}}(\mu).
\]
\end{lemma}

\begin{proof}
The displayed decomposition follows directly from the descendant identity
\[
\mu(I_u) = A_uZ_\infty^{(u)},
\]
the atomization constructions in Subsection~\ref{SS:arc-assembly-atomization}, 
the fixed-cutoff decomposition in
Subsection~\ref{SS:arc-assembly-fixed-cutoff}, 
and the summed finite-level martingale identity in Subsection~\ref{SS:arc-assembly-skeleton}.

By Lemmas~\ref{L:arc-limiting-atomization} and \ref{L:arc-level-atomization}, 
the first and third bracketed terms are bounded by \(C_\gamma K2^{-h}\) on \(\Omega_K\) 
and by \(C_\gamma K_1 2^{-h}\) on \(\mathcal{Z}_{K_1}\), respectively, uniformly over the grid.
By Lemma~\ref{L:arc-terminal-descendant-control},
\[
\sup_{\xi\in\Lambda_n(K,\varepsilon)}
\left|
\sum_{|u|=n+h}
A_u
\bigl(Z_\infty^{(u)}-1\bigr)
\mathrm{e}^{i\phi_\xi(t_u)}
\right|
\longrightarrow0
\qquad
(n\to\infty)
\]
almost surely on
\(
        \mathcal{Z}_{K_1}\cap\{Z_\infty>0\}.
\)

Furthermore, Lemma~\ref{L:arc-safe-region-control} gives
\[
        \limsup_{n\to\infty}
        \sup_{\xi\in\Lambda_n(K,\varepsilon)}
        |E_{\xi,n,h,j}^{\mathrm{safe}}|
        \leq
        C_\gamma\Delta_{C_\gamma2^{-j}}(\mu), 
\]
while \eqref{E:arc-deterministic-forcing-bound} gives
\[
        \sup_{\xi\in\Lambda_n(K,\varepsilon)}
        \left|
        D_{\xi,n,j}^{(0)}
        \right|
        \leq
        C_{\gamma,j}2^{-n}.
\]
Finally, Lemma~\ref{L:arc-oscillatory-skeleton-control} gives
\[
        \sup_{\xi\in\Lambda_n(K,\varepsilon)}
        \left|
        S_{\xi,j,n,h}
        \right|
        \longrightarrow0
\]
almost surely on \(\mathcal{Z}_{K_1}\).

Taking the supremum over \(\xi\in\Lambda_n(K,\varepsilon)\), 
followed by the limit superior as \(n\to\infty\), proves the assertion.
\end{proof}

\subsection{Conclusion of the annular argument}
\label{SS:arc-assembly-conclusion}
\begin{proof}[Proof of Theorem~\ref{T:main-arc-rajchman}]
Fix \(K,K_1\in\mathbb{N}\) and \(\varepsilon\in\mathbb{Q}_{>0}\), and work on
\(
\Omega_K\cap\mathcal{Z}_{K_1}\cap\{Z_\infty>0\}.
\)

Since \(\mu\) is atomless on non-extinction,
Lemma~\ref{L:preliminaries-small-interval-modulus} gives
\[
        \Delta_{C_\gamma2^{-j}}(\mu)
        \longrightarrow0
        \qquad(j\to\infty).
\]
For every pair of integers \(j\geq1\) and \(h\geq0\),
Lemma~\ref{L:arc-full-annular-decomposition} yields
\[
\limsup_{n\to\infty} \sup_{\xi\in\Lambda_n(K,\varepsilon)} |\widehat{\mu_\gamma}(\xi)| 
\leq 
C_\gamma(K+K_1)2^{-h} +
C_\gamma \Delta_{C_\gamma2^{-j}}(\mu).
\]
The left-hand side is independent of \(j\) and \(h\).
Letting first \(j\to\infty\) and then \(h\to\infty\), we obtain
\[
        \sup_{\xi\in\Lambda_n(K,\varepsilon)}
        |\widehat{\mu_\gamma}(\xi)|
        \longrightarrow0
        \qquad(n\to\infty).
\]
Hence, for all sufficiently large \(n\),
\[
        \sup_{\xi\in\Lambda_n(K,\varepsilon)}
        |\widehat{\mu_\gamma}(\xi)|
        <
        \frac{31}{32}\varepsilon.
\]

If \(S_{\gamma,n}>\varepsilon\), however, Lemma~\ref{L:arc-original-grid} gives 
\(\xi\in\Lambda_n(K,\varepsilon)\) such that
\(
|\widehat{\mu_\gamma}(\xi)| > \frac{31}{32}\varepsilon,
\)
a contradiction. 
Thus \(S_{\gamma,n}\leq\varepsilon\) for all sufficiently large \(n\) on
\(
\Omega_K\cap\mathcal{Z}_{K_1}\cap\{Z_\infty>0\}.
\)

Since the preceding almost-sure statements were arranged simultaneously for all \(\varepsilon\in\mathbb{Q}_{>0}\), 
it follows that
\[
S_{\gamma,n} \longrightarrow0 \qquad(n\to\infty)
\]
on
\(
\Omega_K\cap\mathcal{Z}_{K_1}\cap\{Z_\infty>0\}.
\)

Finally,
\(
        \bigcup_{K,K_1\in\mathbb{N}}
        \bigl(\Omega_K\cap\mathcal{Z}_{K_1}\bigr)
\)
has probability one. Therefore, almost surely on
\(\{Z_\infty>0\}=\{\mu([0,1])>0\}\),
\[
        \widehat{\mu_\gamma}(\xi)
        \longrightarrow0
        \qquad(|\xi|\to\infty).
\]
\end{proof}

\section{Jordan curves and pure Rajchman consequences}
\label{S:jordan-cutting}

This section derives the fixed-Jordan-curve theorem from the fixed-arc theorem and records the pure Rajchman consequences in the heavy-tail regime. 
No further oscillatory analysis is needed.

\begin{proof}[Proof of Theorem~\ref{T:main-jordan-rajchman}]
For \(i=0,1\), set \(\Gamma_i=\Gamma\circ\rho_i\).
As observed in Subsection~\ref{SS:preliminaries-curves}, each
\(\Gamma_i\) is a fixed nondegenerate \(C^2\) embedded arc. 
Moreover, Lemma~\ref{L:first-generation-cutting}, applied with \(F=\Gamma\),
gives
\[
        \Gamma_\#\widetilde\mu^{\mathbb T}
        =
        \frac{W_0^{\mathbb T}}2
        (\Gamma_0)_\#\mu^{\mathbb T,(0)}
        +
        \frac{W_1^{\mathbb T}}2
        (\Gamma_1)_\#\mu^{\mathbb T,(1)}.
\]

For each \(i=0,1\), the measure \(\mu^{\mathbb T,(i)}\) is a copy in law of the interval cascade \(\mu\). 
Therefore, Theorem~\ref{T:main-arc-rajchman} gives
\[
        \widehat{(\Gamma_i)_\#\mu^{\mathbb T,(i)}}(\xi)
        \longrightarrow0
        \qquad (|\xi|\to\infty)
\]
almost surely on \(\{\mu^{\mathbb T,(i)}([0,1])>0\}\).
On the complementary extinction event, \(\mu^{\mathbb T,(i)}\) is the zero measure, 
so the same conclusion is trivial. 
Hence, almost surely, the displayed convergence holds simultaneously for \(i=0,1\).

Since \(W_0^{\mathbb T}\) and \(W_1^{\mathbb T}\) are finite almost surely, 
linearity of the Fourier transform yields
\[
\widehat{\Gamma_\#\widetilde\mu^{\mathbb T}}(\xi)
        =
        \frac{W_0^{\mathbb T}}2
        \widehat{(\Gamma_0)_\#\mu^{\mathbb T,(0)}}(\xi)
        +
        \frac{W_1^{\mathbb T}}2
        \widehat{(\Gamma_1)_\#\mu^{\mathbb T,(1)}}(\xi)
        \longrightarrow0
\]
as \(|\xi|\to\infty\). 
In particular, the conclusion holds almost surely on \(\{\widetilde\mu^{\mathbb T}(\mathbb T)>0\}\).
\end{proof}

The three Rajchman theorems are now established. 
By \cite[Corollary~6.2]{CaiChengFangLiQuXiao2026} and \cite[Theorems~1.2 and~1.4]{CaiFangQu2026Curves}, 
the corresponding exact Fourier-dimension formulas hold under the standing minimal Kahane--Peyri\`ere assumptions. 
These external formulas are used only in the following argument and play no role in the proofs of the Rajchman theorems.

For completeness, set
\[
        D_{\mathrm{int}}(W)
        :=
        \sup_{1<q<2}
        \max\left\{
        0,
        2-\frac{2}{q}
        \bigl(1+\log_2\E[W^q]\bigr)
        \right\}
\]
and
\[
        A_{\mathrm{loc}}(W)
        :=
        \sup_{q>1}
        \max\left\{
        0,
        \frac{q-1-\log_2\E[W^q]}{q}
        \right\},
\]
where, in both definitions, the term indexed by \(q\) is interpreted as zero whenever \(\E[W^q]=\infty\). 
Under the standing minimal Kahane--Peyri\`ere assumptions, the exact Fourier-dimension formulas give, 
almost surely on the corresponding non-extinction events,
\[
\dimF\mu = D_{\mathrm{int}}(W),
\]
and, for every prescribed fixed nondegenerate \(C^2\) embedded arc \(\gamma\) and every prescribed fixed nondegenerate \(C^2\) Jordan curve \(\Gamma\),
\[
        \dimF(\gamma_\#\mu)
        =
        \dimF\bigl(
                \Gamma_\#\widetilde{\mu}^{\T}
        \bigr)
        =
        A_{\mathrm{loc}}(W).
\]

\begin{proof}[Proof of Corollary~\ref{C:main-pure-rajchman}]
Under \eqref{E:intro-no-higher-moments}, every term in the suprema defining \(D_{\mathrm{int}}(W)\) and \(A_{\mathrm{loc}}(W)\) is zero.
Theorems~\ref{T:main-interval-rajchman}, \ref{T:main-arc-rajchman}, 
and \ref{T:main-jordan-rajchman} give the Rajchman property in the interval, fixed-arc, 
and fixed-Jordan-curve settings, respectively.
The exact Fourier-dimension formulas recalled above give, almost surely on the corresponding non-extinction events,
\[
        \dimF\mu
        =
        0,
        \qquad
        \dimF(\gamma_\#\mu)
        =
        0,
        \qquad
        \dimF\bigl(
                \Gamma_\#\widetilde{\mu}^{\T}
        \bigr)
        =
        0.
\]
Thus the measures in all three settings are pure Rajchman almost surely on their corresponding non-extinction events.
\end{proof}

\section*{Statements and Declarations}

\subsection*{Funding}

G. C. was partially supported by the National Natural Science Foundation of China (NSFC), grant no.~12371126.
X. F. was partially supported by the National Science and Technology Council, Taiwan, grant no.~114-2115-M-A49-003-MY3.

\subsection*{Competing interests}

The authors declare that they have no competing interests.

\subsection*{Use of artificial intelligence}

The authors used an artificial-intelligence tool for language editing and
\LaTeX\ formatting; all mathematical content was written, checked, and approved
by the authors.


\begin{thebibliography}{99}

\bibitem{AlgomRodriguezHertzWang2026}
Algom, A., Rodriguez Hertz, F., Wang, Z.:
Spectral gaps and Fourier decay for self-conformal measures on the plane.
Trans. Am. Math. Soc. 379, 3953--3991 (2026).
\url{https://doi.org/10.1090/tran/9507}.

\bibitem{BakerBanaji2025}
Baker, S., Banaji, A.:
Polynomial Fourier decay for fractal measures and their pushforwards.
Math. Ann. 392, 209--261 (2025).
\url{https://doi.org/10.1007/s00208-025-03091-z}.

\bibitem{Barral2014Mandelbrot}
Barral, J.:
Mandelbrot cascades and related topics.
In: Feng, D.-J., Lau, K.-S. (eds.)
Geometry and Analysis of Fractals,
Springer Proceedings in Mathematics \& Statistics, vol.~88,
pp.~1--45. Springer, Berlin, Heidelberg (2014).
\url{https://doi.org/10.1007/978-3-662-43920-3_1}.

\bibitem{Bremont2021}
Br\'emont, J.:
Self-similar measures and the Rajchman property.
Ann. Henri Lebesgue 4, 973--1004 (2021).
\url{https://doi.org/10.5802/ahl.94}.

\bibitem{CaiChengFangLiQuXiao2026}
Cai, Y., Cheng, G., Fang, X., Li, M., Qu, H., Xiao, C.:
Exact Fourier dimensions of dyadic Mandelbrot cascades under minimal
integrability.
Preprint, arXiv:2606.08683 [math.PR] (2026).
\url{https://doi.org/10.48550/arXiv.2606.08683}.

\bibitem{CaiFangQu2026Curves}
Cai, Y., Fang, X., Qu, H.:
Exact Fourier dimensions of dyadic Mandelbrot cascades on curves of
nonvanishing curvature under minimal integrability.
Preprint, arXiv:2606.11758 [math.PR] (2026).
\url{https://doi.org/10.48550/arXiv.2606.11758}.

\bibitem{ChenHanQiuWang2024}
Chen, X., Han, Y., Qiu, Y., Wang, Z.:
Harmonic analysis of Mandelbrot cascades---in the context of vector-valued
martingales.
Preprint, arXiv:2409.13164 [math.PR] (2024).
\url{https://doi.org/10.48550/arXiv.2409.13164}.

\bibitem{ChenHanQiuWang2025CR}
Chen, X., Han, Y., Qiu, Y., Wang, Z.:
The Mandelbrot--Kahane problem of Beno\^{\i}t Mandelbrot model of turbulence.
C. R. Math. 363, 35--41 (2025).
\url{https://doi.org/10.5802/crmath.697}.

\bibitem{ChenLiSuomala2025}
Chen, C., Li, B., Suomala, V.:
Fourier dimension of Mandelbrot multiplicative cascades.
Commun. Math. Phys. 406, 182 (2025).
\url{https://doi.org/10.1007/s00220-025-05354-x}.

\bibitem{Freedman1975}
Freedman, D.A.:
On tail probabilities for martingales.
Ann. Probab. 3, 100--118 (1975).
\url{https://doi.org/10.1214/aop/1176996452}.

\bibitem{GarbanVargas2026}
Garban, C., Vargas, V.:
Harmonic analysis of Gaussian multiplicative chaos on the circle.
Probab. Theory Relat. Fields (2026).
\url{https://doi.org/10.1007/s00440-026-01497-7}.

\bibitem{Kahane1985SomeRandomSeries}
Kahane, J.-P.:
Some Random Series of Functions, 2nd edn.
Cambridge Studies in Advanced Mathematics, vol.~5.
Cambridge University Press, Cambridge (1985).

\bibitem{Kahane1987PositiveMartingales}
Kahane, J.-P.:
Positive martingales and random measures.
Chin. Ann. Math. Ser. B 8, 1--12 (1987).

\bibitem{Kahane1993}
Kahane, J.-P.:
Fractals and random measures.
Bull. Sci. Math. 117, 153--159 (1993).

\bibitem{KahanePeyriere1976}
Kahane, J.-P., Peyri\`ere, J.:
Sur certaines martingales de Beno\^{\i}t Mandelbrot.
Adv. Math. 22, 131--145 (1976).
\url{https://doi.org/10.1016/0001-8708(76)90151-1}.

\bibitem{LiSahlsten2020}
Li, J., Sahlsten, T.:
Fourier transform of self-affine measures.
Adv. Math. 374, 107349 (2020).
\url{https://doi.org/10.1016/j.aim.2020.107349}.

\bibitem{LiSahlsten2022}
Li, J., Sahlsten, T.:
Trigonometric series and self-similar sets.
J. Eur. Math. Soc. 24, 341--368 (2022).
\url{https://doi.org/10.4171/JEMS/1102}.

\bibitem{Lyons1995Rajchman}
Lyons, R.:
Seventy years of Rajchman measures.
J. Fourier Anal. Appl. Special Issue, 363--377 (1995).

\bibitem{LyonsPemantlePeres1995}
Lyons, R., Pemantle, R., Peres, Y.:
Conceptual proofs of \(L\log L\) criteria for mean behavior of branching
processes.
Ann. Probab. 23, 1125--1138 (1995).
\url{https://doi.org/10.1214/aop/1176988176}.

\bibitem{Mandelbrot1974}
Mandelbrot, B.B.:
Intermittent turbulence in self-similar cascades:
divergence of high moments and dimension of the carrier.
J. Fluid Mech. 62, 331--358 (1974).
\url{https://doi.org/10.1017/S0022112074000711}.

\bibitem{Mandelbrot1976}
Mandelbrot, B.B.:
Intermittent turbulence and fractal dimension:
kurtosis and the spectral exponent \(5/3+B\).
In: Temam, R. (ed.)
Turbulence and Navier--Stokes Equations,
Lecture Notes in Mathematics, vol.~565,
pp.~121--145. Springer, Berlin, Heidelberg (1976).
\url{https://doi.org/10.1007/BFb0091452}.

\bibitem{Mattila2015}
Mattila, P.:
Fourier Analysis and Hausdorff Dimension.
Cambridge Studies in Advanced Mathematics, vol.~150.
Cambridge University Press, Cambridge (2015).
\url{https://doi.org/10.1017/CBO9781316227619}.

\bibitem{Rapaport2022}
Rapaport, A.:
On the Rajchman property for self-similar measures on \(\mathbb{R}^{d}\).
Adv. Math. 403, 108375 (2022).
\url{https://doi.org/10.1016/j.aim.2022.108375}.

\bibitem{RhodesVargas2014}
Rhodes, R., Vargas, V.:
Gaussian multiplicative chaos and applications: a review.
Probab. Surv. 11, 315--392 (2014).
\url{https://doi.org/10.1214/13-PS218}.

\bibitem{RyouSuomala2026}
Ryou, D., Suomala, V.:
Fourier dimension of Mandelbrot cascades on planar curves.
Preprint, arXiv:2603.25615 [math.PR] (2026).
\url{https://doi.org/10.48550/arXiv.2603.25615}.

\bibitem{ShmerkinSuomala2018}
Shmerkin, P., Suomala, V.:
Spatially independent martingales, intersections, and applications.
Mem. Am. Math. Soc. 251(1195), v+102 pp. (2018).
\url{https://doi.org/10.1090/memo/1195}.

\bibitem{VershyninHDP}
Vershynin, R.:
High-Dimensional Probability: An Introduction with Applications in Data Science.
Cambridge Series in Statistical and Probabilistic Mathematics, vol.~47.
Cambridge University Press, Cambridge (2018).
\url{https://doi.org/10.1017/9781108231596}.

\end{thebibliography}
\end{document}